\documentclass[preprint,12pt,3p]{elsarticle}
\usepackage{verbatim}
\usepackage{amsmath,amsthm,amsfonts,amssymb}
\usepackage{algorithm}
\usepackage{algpseudocode}
\usepackage{cases}
\usepackage{hyperref}
\usepackage{multirow}
\usepackage[active,new,old]{correct} 
\usepackage{amsmath,amsfonts,amssymb}
\usepackage[normalem]{ulem}
\usepackage{xcolor,sectsty}
\usepackage{natbib}
\usepackage{stackengine}
\usepackage{subfigure}
\newcommand{\kronecker}{\raisebox{1.2pt}{\ensuremath{\:\otimes\:}}}
\newcommand{\core}{\scriptsize\mbox{\textcircled{\#}}}

\definecolor{astral}{RGB}{46,116,181}
\subsectionfont{\color{astral}}
\sectionfont{\color{astral}}
\newtheorem{theorem}{Theorem}[section]
\newtheorem{lemma}[theorem]{Lemma}
\newtheorem{corollary}[theorem]{Corollary}
\newtheorem{proposition}[theorem]{Proposition}
\newtheorem{definition}[theorem]{Definition}
\newtheorem{example}[theorem]{Example}
\newtheorem{remark}[theorem]{Remark}
\definecolor{darkslategray}{rgb}{0.18, 0.31, 0.31}
\definecolor{warmblack}{rgb}{0.0, 0.26, 0.26}
\journal{arxiv}

\newcommand{\C}{{\mathbb C}}

\newcommand{\mc}[1]{\mathcal {#1}}

\newcommand{\rg}{{\mathscr{R}}}
\newcommand{\nl}{{\mathscr{N}}}
\newcommand{\n}{{*_n}}

\newcommand{\m}{{*_m}}

\usepackage{mathrsfs}

\begin{document}

\begin{frontmatter}

\title{ \textcolor{warmblack}{\bf 
 Reverse-order law for core inverse of tensors
 }}

\author{Jajati Kesahri Sahoo $^\dagger$ and   Ratikanta Behera$^*$
}

\address{$^\dagger$
Department of Mathematics,\\
BITS Pilani, K.K. Birla Goa Campus, Goa, India
\\\textit{E-mail}: \texttt{jksahoo\symbol{'100}goa.bits-pilani.ac.in }
\vspace{.3cm}\\
$^*$
Department of Mathematics and Statistics,\\
Indian Institute of Science Education and Research Kolkata,\\
 Nadia, West Bengal, India.\\
\textit{E-mail}: \texttt{ratikanta@iiserkol.ac.in}

}
\begin{abstract}
\textcolor{warmblack}{
The notion of the core inverse of tensors with the Einstein product was introduced, very recently. This paper we establish some sufficient and necessary conditions for reverse-order law of this inverse. Further, we present new results related to the mixed-type reverse-order law for core inverse. In addition to these, we discuss core inverse solutions of multilinear systems of tensors via the Einstein product. The prowess of the inverse is demonstrated for solving Poisson problem in the multilinear system framework.
}
\end{abstract}


\end{frontmatter}

\section{Introduction and motivation}\label{sec1}
Let $\mathbb{C}^{I_1\times\cdots\times I_N}$  be the set of  order
$N$ and dimension $I_1 \times \cdots \times I_N$ tensors over the complex
field $\mathbb{C}$ and the entry of tensor $\mc{A}$ is denoted by $a_{i_1...i_N}$.
Here $\mc{A} \in \mathbb{C}^{I_1\times\cdots\times I_N}$ is a multiway array with  $N$-th order tensor, and $I_1, I_2, \cdots, I_N$ are dimensions of the first, second,  $\cdots$ , $N$th way, respectively. Note that throughout the paper, tensors are represented in calligraphic letters like  $\mc{A}$, and the notation $\mc{A}_{i_1...i_n}$ represents the scalars. The Einstein product (\cite{ein}) $ \mc{A}\n\mc{B} \in \mathbb{C}^{I_1\times\cdots\times
I_N \times J_1 \times\cdots\times J_M }$ of tensors $\mc{A} \in \mathbb{C}^{I_1\times\cdots\times I_N \times K_1
\times\cdots\times K_N }$ and $\mc{B} \in
\mathbb{C}^{K_1\times\cdots\times K_N \times J_1 \times\cdots\times
J_M }$   is defined by the operation $\n$ via
\begin{equation}\label{Eins}
(\mc{A}\n\mc{B})_{i_1...i_nj_1...j_m}
=\displaystyle\sum_{k_1,\ldots,k_n}\mc{A}_{{i_1...i_n}{k_1...k_n}}\mc{B}_{{k_1...k_n}{j_1...j_m}}\in \mathbb{C}^{I_1\times\cdots\times I_n \times J_1 \times\cdots\times J_m }.
\end{equation}
The Einstein products ideally suited to addressing the problem of finding multilinear structure in multiway data-sets. Such multiscale issues are encountered in many fields of practical interest \cite{RD, ein, lai, sun}.

In connection with tensors and multilinear systems, Brazell {\it et al.} \cite{BraliNT13} introduced the notion of the
ordinary tensor inverse,  as follows. A tensor $\mc{X} \in
\mathbb{C}^{I_1\times\cdots\times I_N \times I_1 \times\cdots\times
I_N}$ is called  the  {\it inverse} of $\mc{A}\in
\mathbb{C}^{I_1\times\cdots\times I_N \times I_1 \times\cdots\times
I_N}$  if it satisfies $\mc{A}\n\mc{X}=\mc{X}\n\mc{A}=\mc{I}$. It
is denoted  by $\mc{A}^{-1}$. As a consequence, most of the current efforts are focused on developing  different type of inverses of tensors depending on the application \cite{BraliNT13, ishteva2011best, ShiWeL13}. Sun et al. \cite{sun} introduced Moore-Penrose inverse of tensors and discussed the existence and uniqueness of the inverse with the Einstein product. In this context, Behera and Mishra \cite{RD} discussed ${i}$-inverses (for $i = 1, 2, 3, 4$) of even-order tensors and general solutions of multilinear systems. Subsequently, the LU and the Schur decompositions in the tensor were discussed in \cite{liang2019}. Further, solutions of a multilinear system represented by generalized inverses of tensors were discussed in  \cite{sun2018}. Using such theory of the Einstein product,  Ji and Wei \cite{weit2}  introduced weighted Moore-Penrose inverse of tensors and presented a few characterization of the least-squares solutions to a multilinear system. Further, full rank decomposition \cite{liangBing2019} of arbitrary-order tensors  using reshape operation is discussed in \cite{Behera2018}. At the beginning of this work, Stanimirovic at el. \cite{stan} discussed the effective definition of the tensor rank and outer inverses of tensors. In addition to this, the Drazin inverse of even-order tensors and its applications via the Einstein product is investigated in \cite{RAJ, Ji18}. Recently, Liang and Zheng \cite{liang2018} defined an iterative algorithm for solving Sylvester tensor equation using such theory. 

On the other hand, the concept of tensor inversion technique has opened new perspectives for solving multilinear systems across many fields in science and engineering \cite{ SilL08, kolda, Kru77, LalmV00,  liang2018, SidLaP17}). The inversion of tensor impinges upon a definition of tensor-tensor multiplication, and the treatment of multilinear system is a straightforward task due to the nature of the Einstein product. However, this becomes particularly challenging for problems having singular or arbitrary-order tensor.  Generalized inverses of tensors take advantage to do this \cite{ RD, RAJ, BraliNT13, Ji18}.  In this context, Sahoo at al. \cite{JRPKH} introduced the core inverses of tensors and discussed some characterizations of inverses. At last, they posed the open question regarding reverse-order law for core inverse of the tensors. It will be more interesting if we study the reverse-order law for the core inverse of tensors.

It is well known that $(\mc{A}\n \mc{B})^{-1} = \mc{B}^{-1}\n \mc{A}^{-1},$ where  $\mc{A}\in \mathbb{C}^{I_1\times\cdots\times I_N \times I_1
\times\cdots\times I_N }$ and $\mc{B}\in \mathbb{C}^{I_1\times\cdots\times I_N \times I_1
\times\cdots\times I_N }$ are a pair of invertible tensors such that their Einstein product $\mc{A}\n\mc{B}$ is invertible. This equality is called the reverse-order law for the ordinary inverse. In general, this equality doesn’t hold if we replace the ordinary inverse by the generalized inverse of tensors. This reverse-order law is a class of interesting problems that are fundamental in the theory of generalized inverses of tensors. In this context, Panigrahy at al \cite{RDP} first gave a necessary and sufficient condition of the reverse-order law
for the Moore-Penrose inverse of tensors via the Einstein product. Since then, many authors studied the reverse-order law for various classes of generalized inverses of tensors \cite{Mispa18, RDP,  panigrahy2019}.

Indeed, the reverse-order law, first studied a necessary and sufficient condition for the Moore-Penrose inverse in from of rectangular matrices in \cite{greville1966}. Baskett and Katz \cite{baskett1969} then discussed the same theory for $EP_r$ matrices, where $A\in {\C}^{n\times n}$ of rank $r$ is called $EP_r$, if $A$ and $A^*$, the conjugate transpose of $A$, have the same null spaces.  Further, Deng \cite{deng2011} investigated some necessary and sufficient conditions of the reverse-order law for the group inverse of linear bounded operators on Hilbert spaces. The reverse-order law was also studied for  other generalized inverses of matrices (\cite{Barwick1974}, \cite{Cao2004}, \cite{sun1998}). Under the influence of the vast work on the reverse-order law and solution of multilinear systems. Here we discuss reverse-order law for core inverse of tensors and a few characterizations of the multilinear systems. The results are new even in the case of matrices.

We organize the paper as follows. In the next section, we discuss some notations and definitions along with a few preliminary results, which are helpful in proving the main results. In Section 3 contains a few necessary and sufficient conditions of the reverse-order law for the core inverses of tensors via the Einstein product. Further, in section 4, we prove several results concerning mixed-type of the reverse-order law. General solutions of multilinear systems using core inverse of the tensors are addressed in section 5, and the conclusion is outlined in section 6.

\section{Preliminaries}
\subsection{Definitions and terminology}
For convenience, we first briefly explain some additional notations to increase the efficiency of the presentation, as follows.
\begin{equation*}
 \textbf{M}(k) = M_1\times M_2 \times\cdots\times M_k,~ \textbf{m}(k) = \{m_1,m_2,\ldots, m_k | 1\leq m_j \leq M_j,~ j = 1, \ldots , k\}.
\end{equation*}
In view of the above notation, the tensor $\mc{A} =
(a_{m_1,\ldots, m_k} )_{1\leq m_j\leq M_j }, j = 1,\ldots, k$ will be shortly denoted by $\mc{A} = (a_{\textbf{m}(k)}).$ Further, the
following notation is useful:
\begin{equation*}
 \sum_{\textbf{m}(k)}=\sum_{m_1,\ldots, m_k}=
\sum_{m_1=1}^{M_1}\cdots\sum_{m_1=1}^{M_k}
\end{equation*}
Accordingly, the tensor $\mathcal{A} \in \mathbb{C}^{M_{1} \times \cdots \times M_{{m}} \times N_{1} \times \cdots \times N_n}$ will be denoted in a simpler form as $\mathcal{A} \in \mathbb{C}^{\textbf{M}(m) \times \textbf{N}(n)}$.
For a tensor $~\mc{A}=(a_{\textbf{m}(k),\textbf{n}(k)}\in \mathbb{C}^{\textbf{M}(k) \times \textbf{N}(k)},$ the {\it transpose} of $\mc{A}$ is denoted by $\mc{A}^T$ and defined as $\mc{A}^T=(a_{\textbf{n}(k),\textbf{m}(k)}\in \mathbb{C}^{\textbf{N}(k) \times \textbf{M}(k)}$.
 Further, we denote $\mc{O}$ by {\it zero tensor} and $\mc{I}$ is the {\it identity} tensor.
 A tensor $\mc{A}\in
\mathbb{C}^{\textbf{N}(n) \times \textbf{N}(n)}$ is {\it Hermitian}  if  $\mc{A}=\mc{A}^*,$ {\it skew-Hermitian} if $\mc{A}= - \mc{A}^*$, and {\it orthogonal} if $\mc{A}\m\mc{A}^*= \mc{A}^*\n \mc{A}=\mc{I}$. Let us recall the definition of the Moore-Penrose inverse of a tensor which was introduced in \cite{sun}, as follows. 
\begin{definition}( Definition 2.2,  \cite{sun})
Let $\mc{A}\in \mathbb{C}^{\textbf{M}(m) \times \textbf{N}(n)}$. The tensor  $\mc{X}\in \mathbb{C}^{\textbf{N}(n) \times \textbf{M}(m)}$ satisfying
\begin{center}
     \textup{(1)} $\mc{A}\n\mc{X}\m\mc{A}=\mc{A};$ $\textup{(2) }\mc{X}\m\mc{A}\n\mc{X}=\mc{X};$
    \textup{(3)} $(\mc{A}\n\mc{X})^{*}=\mc{A}\n\mc{X};$ \textup{(4)} $(\mc{X}\m\mc{A})^{*}=\mc{X}\m\mc{A}$,
\end{center}
is called Moore-Penrose inverse of $\mc{A}$ and denoted by  $\mc{A}^{\dagger}$.
\end{definition}
Further, the authors of \cite{sun} also discussed the existence and uniqueness of the inverse. Note that, the Moore-Penrose inverse always exists and unique.

Now, we recall the index of a tensor which was first discusses in \cite{Ji18}. The index of a tensor $\mc{A}\in\mathbb{C}^{\textbf{N}(n) \times \textbf{N}(n)}$ is defined as the smallest positive integer $k,$ such that $\rg(\mc{A}^k)=\rg(\mc{A}^{k+1})$ and it is denoted by ind$(\mc{A})$.  If $k=1,$ then the tensor is called index one or group, or core tensor.  Let the tensor $\mc{A}\in\mathbb{C}^{\textbf{N}(n) \times \textbf{N}(n)}$ and ind$(\mc{A})=k$, then the Drazin inverse of $\mc{A}$ is defined in \cite{Ji18} as follows.
\begin{definition}
Let $\mc{A}\in \mathbb{C}^{\textbf{N}(n) \times \textbf{N}(n)}$ and ind$(\mc{A})=k$. A tensor $\mc{X}\in \mathbb{C}^{\textbf{N}(n) \times \textbf{N}(n) }$ satisfying
\begin{center}
\textup{(1$^k$)}~$\mc{A}^{k+1}\n\mc{X}=\mc{A}^k,~\textup{(2)}~\mc{X}\n\mc{A}\n\mc{X}=\mc{X},~\textup{(5)}~\mc{X}\n\mc{A}=\mc{A}\n\mc{X},$
\end{center}
 is called Drazin inverse of $\mc{A}$.
\end{definition}
The Drazin inverse is uniques if exists and denoted by $\mc{A}^D$. In particular, when $k=1$, the Drzain inverse is called group inverse of $\mc{A}$ and denoted by $\mc{A}^{\#}$. Subsequently, the notion of core inverse for complex tensors was introduced by Sahoo et. al. \cite{JRPKH}, recently, as follows.
\begin{definition}\label{coredef}
Let $\mc{A}\in \mathbb{C}^{\textbf{N}(n) \times \textbf{N}(n) }$. A tensor $\mc{X}\in \mathbb{C}^{\textbf{N}(n) \times \textbf{N}(n)}$ satisfying
\begin{center}
  \textup{(3)} $(\mc{A}\n\mc{X})^*=\mc{A}\n\mc{X};$  \textup{(6)} $\mc{X}\n\mc{A}^2=\mc{A};$ \textup{(7)} $\mc{A}\n\mc{X}^2=\mc{X}$,
\end{center}
  is called core inverse of $\mc{A}$.
\end{definition}
The core inverse is unique whenever it exists and denoted by $\mc{A}^{\core}$. If a tensor $\mc{A}$ is core tensors then its core inverse always exist and we call the tensor $\mc{A}$, is core invertible. Further, a tensor $\mc{A}\in \mathbb{C}^{\textbf{N}(n) \times \textbf{N}(n) }$ is called EP \cite{sahoo2019} if $\mc{A}\n\mc{A}^\dagger=\mc{A}^\dagger\n\mc{A}$.

Let $S=\{1,2,3,4,5,6,7\}$. If a tensor $\mc{X}$ satisfies equation (i), where $i\in S$, then $\mc{X}$ is called $\{i\}$-inverse of $\mc{A}$. The set of $\{i\}$-inverse of the tensor $\mc{A}$ is denoted by $\mc{A}\{i\}$. For example, $\mc{A}\{1\}$, $\mc{A}\{1,2\}$ and $\mc{A}\{1,3\}$ respectively for set of generalized inverse, generalized reflexive inverse of $\mc{A}$,  and $\{1,3\}$ inverse of $\mc{A}.$  We use the notation $\mc{A}^{(i)}$, is a fixed element of $\mc{A}\{i\}$. For example, $\mc{A}^{(1)}$, $\mc{A}^{(1,3)}$ respectively one element of $\mc{A}\{1\}$ and $\mc{A}\{1,3\}$. The range and null space of a tensor  $\mc{A}\in \mathbb{C}^{\textbf{M}(m)\times\textbf{N}(n)}$ was introduced in \cite{Ji18, stan}, and defined as
$$
\rg(\mc{A}) = \left\{\mc{A}\n\mc{X}:~\mc{X}\in\mathbb{C}^{\textbf{N}(n)}\right\}\mbox{ and } \nl(\mc{A})=\left\{\mc{X}:~\mc{A}\n\mc{X}=\mc{O}\in\mathbb{C}^{\textbf{M}(m)}\right\}.
$$
We conclude this subsection with the definition of Kronecker product and one result using the product, which are essential to prove our main results.
\begin{definition}{\cite{sun}}
The Kronecker product of $\mc{A} \in
\mathbb{C}^{\textbf{I}(N)\times\textbf{J}(N)}$
 and $\mc{B} \in \mathbb{C}^{\textbf{K}(M)\times\textbf{L}(M)}$, denoted by
 $\mc{A}\otimes\mc{B} = (a_{i_1...i_Nj_1...j_N}\mc{B}),$ is a `Kr-block tensor' whose $(t_1, t_2)$
  subblock is $(a_{i_1...i_Nj_1...j_N}\mc{B})$ obtained via multiplying all the entries of $\mc{B}$ by
  a constant $a_{i_1...i_Nj_1...j_N},$ where\\
    $t_1=i_N + \displaystyle\sum_{K=1}^{N-1} \left[ (i_K - 1) \displaystyle\prod_{L=K+1}^{N} I_L \right]$ and
   $t_2=j_N + \displaystyle\sum_{K=1}^{N-1} \left[ (j_K - 1) \displaystyle\prod_{L=K+1}^{N} J_L
   \right]$.
\end{definition}
\begin{proposition} [Proposition 1.3, \cite{RD} and Proposition 2.3, \cite{sun}]\label{Krop}
Let $\mc{A},~\mc{B},~\mc{C},\mbox{ and }\mc{D} \in \mathbb{C}^{\textbf{N}(n)\times\textbf{N}(n)}$. Then
\item[(a)] $(\mc{A}\otimes \mc{B})^* = \mc{A}^* \otimes \mc{B}^*$;
\item[(b)] $\mc{A}\otimes(\mc{B}\otimes\mc{C})
 = (\mc{A}\otimes\mc{B})\otimes \mc{C}$;
\item[(c)] $(\mc{A}\otimes\mc{B})\n(\mc{C}\otimes\mc{D})=(\mc{A}\n\mc{C})\otimes(\mc{B}\n\mc{D})$.
\end{proposition}

\subsection{Prerequisite results}
Now, we discuss some preliminary results, which will be used in our subsequent sections. Out of which, few known results directly collected from literature. We first recall the result based on the Drazin inverse, as follows.
\begin{lemma}[Theorem 3.13, \cite{RAJ}]\label{lm2.4}
Let $\mc{A}$, $\mc{B}\in\C^{\textbf{N}(n) \times \textbf{N}(n)}$ and ind$(\mc{A}\n\mc{B})=k$. Then
$(\mc{B}\n\mc{A})^D=\mc{B}((\mc{A}\n\mc{B})^D)^2\n\mc{A}$.
\end{lemma}
Now we present a few impressive results from \cite{JRPKH} in which the core inverse of a tensor in linked with other types of inverses, i.e., Moore-Penrose inverse, group inverse, and EP.
\begin{lemma}[Theorem 3.3, \cite{JRPKH}]\label{lm2.5}
Let $\mc{A}\in\C^{\textbf{N}(n) \times \textbf{N}(n)}$
be a core tensor. If there exists a tensor $\mc{X}:=\mc{A}^{(1,3)}\in\C^{\textbf{N}(n) \times \textbf{N}(n)}$
such that $\mc{X}\in\mc{A}\{1,3\}$, then $\mc{A}^{\core}=\mc{A}^{\#}\n\mc{A}\n\mc{X}$.
\end{lemma}
\begin{lemma}[Theorem 3.2, 3.5, \cite{JRPKH}]\label{lm2.6}
Let  $\mc{A}\in\C^{\textbf{N}(n) \times \textbf{N}(n)}$
be a core tensor. Then the following statements are equivalent:
\begin{enumerate}
    \item[(a)] $\mc{A}$ is EP;
    \item[(b)] $\mc{A}^{\core}=\mc{A}^{\dagger}=\mc{A}^{\#}$;
    \item[(c)]  $\mc{A}\n\mc{A}^{\core}=\mc{A}^{\core}\n\mc{A}$.
\end{enumerate}
\end{lemma}

Further, the authors of \cite{stan}, proved the following result for the range space.
\begin{lemma}[Lemma 2.2, \cite{stan}]\label{lm2.7}
Let  $\mc{A}\in \mathbb{C}^{\textbf{M}(m)\times \textbf{N}(n)}$,  $\mc{B}\in \mathbb{C}^{\textbf{M}(m)\times \textbf{K}(k)}.$ Then $\mathfrak{R}(\mc{B})\subseteq\mathfrak{R}(\mc{A})$ if and only if there exists  $\mc{U}\in \mathbb{C}^{\textbf{N}(n)\times \textbf{K}(k)}$ such that
$\mc{B}=\mc{A}\n\mc{U}$.
\end{lemma}
Using Lemma \ref{lm2.7}, we can easily show the following result.
\begin{lemma}\label{lm2.8}
Let  $\mc{A}\in \mathbb{C}^{\textbf{N}(n)\times \textbf{N}(n)}$. Then $\mc{A}^{\#}$ exists if and only if $\mc{A}=\mc{A}^2\n\mc{X}=\mc{Y}\n\mc{A}^2$ for some $\mc{X}$, $\mc{Y}\in\C^{\textbf{M}(m)\times \textbf{N}(n)}$.
\end{lemma}
\begin{corollary}\label{cor2.9}
Let  $\mc{A}\in \mathbb{C}^{\textbf{N}(n)\times \textbf{N}(n)}$ and ind$(\mc{A})=1$. Then $\mc{A}^{\#}=\mc{Y}\n\mc{A}\n\mc{X}=\mc{A}\n\mc{X}^2=\mc{Y}^2\n\mc{A}$ for some $\mc{X}$, $\mc{Y}\in\C^{\textbf{M}(m)\times \textbf{N}(n)}$.
\end{corollary}
\begin{proof}
Since $\mc{A}^{\#}$ exists, by Lemma \ref{lm2.8}, we have $\mc{Y}\n\mc{A}\n\mc{X}=\mc{Y}\n\mc{A}^2\n\mc{X}\n\mc{X}=\mc{A}\n\mc{X}^2$ and $\mc{Y}\n\mc{A}\n\mc{X}=\mc{Y}\n\mc{Y}\n\mc{A}^2\n\mc{X}=\mc{Y}^2\n\mc{A}$. To claim the result, it is enough to show $\mc{Z}=\mc{Y}\n\mc{A}\n\mc{X}$ is the group inverse of $\mc{A}$. Now
\begin{eqnarray*}
\mc{A}\n\mc{Z}\n\mc{A}&=& \mc{A}\n\mc{Y}\n\mc{A}\n\mc{X}\n\mc{A}=\mc{A}\n\mc{Y}\n\mc{A}^2\n\mc{X}^2\n\mc{A}={A}^2\n\mc{X}\n\mc{X}\n\mc{A}\\
&=&{Y}\n\mc{A}^2\n\mc{X}\n\mc{A}={Y}\n\mc{A}^2=\mc{A},
\end{eqnarray*}
\begin{eqnarray*}
\mc{Z}\n\mc{A}\n\mc{Z}&=&\mc{Y}\n\mc{A}\n\mc{X}\n\mc{A}\n\mc{Y}\n\mc{A}\n\mc{X}=\mc{Y}\n\mc{A}\n\mc{X}\n\mc{A}\n\mc{Y}\n\mc{A}^2\n\mc{X}^2\\
&=&\mc{Y}\n\mc{A}\n\mc{X}\n\mc{A}^2\n\mc{X}^2=\mc{Y}\n\mc{A}\n\mc{X}\n\mc{A}\n\mc{X}=\mc{Y}^2\n\mc{A}^2\n\mc{X}\n\mc{A}\n\mc{X}\\
&=&\mc{Y}^2\n\mc{A}^2\n\mc{X}=\mc{Y}\n\mc{A}\n\mc{X}=\mc{Z},\mbox{ and }
\end{eqnarray*}
\begin{eqnarray*}
\mc{Z}\n\mc{A}&=&\mc{Y}\n\mc{A}\n\mc{X}\n\mc{A}=\mc{Y}^2\n\mc{A}^2=\mc{Y}^2\n\mc{A}^2\n\mc{X}\n\mc{A}=\mc{Y}\n\mc{A}\n\mc{X}\n\mc{A}=\mc{Z}\n\mc{A}.
\end{eqnarray*}
Hence $\mc{Z}$ is the group inverse of $\mc{A}$.
\end{proof}
If $\mc{X}$ is $\{1,3\}$-inverse of $\mc{A}$, then $\mc{A}=\mc{A}\n\mc{X}\n\mc{A}=\mc{X}^*\n\mc{A}^*\n\mc{A}$. Conversely, if $\mc{A}=\mc{X}^*\n\mc{A}^*\n\mc{A}$ for some $\mc{X}$, then we can easily get
\begin{equation*}
 (\mc{A}\n\mc{X})^*=((\mc{A}\n\mc{X})^*\n\mc{A}\n\mc{X})^*=(\mc{A}\n\mc{X})^*\n\mc{A}\n\mc{X}=\mc{A}\n\mc{X}, \mbox{ and }
\end{equation*}
$\mc{A}\n\mc{X}\n\mc{A}=(\mc{A}\n\mc{X})^*\n\mc{A}\n\mc{X}\n\mc{A}\n\mc{A}=(\mc{A}\n\mc{X})^*\n(\mc{A}\n\mc{X})^*\n\mc{A}=(\mc{A}\n\mc{X})^*\n\mc{A}=\mc{A}$. Therefore, we obtain the following result.
\begin{lemma}\label{lm2.10}
Let $\mc{A}\in \mathbb{C}^{\textbf{M}(m)\times \textbf{N}(n)}$ and $\mc{X}\in \mathbb{C}^{\textbf{N}(n)\times \textbf{M}(m)}$. Then $\mc{X}$ is $\{1,3\}$-inverse of $\mc{A}$ if and only if $\mc{A}=\mc{X}^*\n\mc{A}^*\n\mc{A}$.
\end{lemma}
Now we recall a characterization of the core-tensor, as follows.
\begin{lemma}[Proposition 3.1 and Corollary 3.7, \cite{JRPKH}]\label{lm2.11}
Let $\mc{A}\in\C^{\textbf{N}(n)\times \textbf{N}(n)}$ be a core tensor. If a tensor $\mc{X}\in\C^{\textbf{N}(n)\times \textbf{N}(n)}$ satisfying any one of the condition
\begin{enumerate}
    \item[(a)] $\mc{A}\n\mc{X}\n\mc{A}=\mc{A}$, $(\mc{A}\n\mc{X})^*=\mc{A}\n\mc{X}$, and $\mc{A}\n\mc{X}^2=\mc{X}$,
    \item[(b)] $\mc{X}\n\mc{A}\n\mc{X}=\mc{X}$, $(\mc{A}\n\mc{X})^*=\mc{A}\n\mc{X}$, and $\mc{X}\n\mc{A}^2=\mc{A}$,
\end{enumerate}
then $\mc{X}$ is the core inverse of $\mc{A}$.
\end{lemma}

Next, we discuss a few results based on commutative of two tensors.
\begin{lemma}[Theorem 3.11 (a) when $k=1$, \cite{RAJ}]\label{lm2.12}
Let $\mc{A},~\mc{B}\in\C^{\textbf{N}(n)\times \textbf{N}(n)}$ and $ind(\mc{A})=ind(\mc{B})=1$. If $\mc{A}\n\mc{B}=\mc{B}\n\mc{A}$, then  $\mc{A}^{\#}\n\mc{B}=\mc{B}\n\mc{A}^{\#}$ and $\mc{A}\n\mc{B}^{\#}=\mc{B}^{\#}\n\mc{A}$. Further, $(\mc{A}\n\mc{B})^{\#}=\mc{B}^{\#}\n\mc{A}^{\#}$.
\end{lemma}
\begin{lemma}\label{lm2.13}
Let $\mc{A},~\mc{X}\in \mathbb{C}^{\textbf{N}(n)\times \textbf{N}(n)}$. If $\mc{A}\n\mc{X}=\mc{X}\n\mc{A}$ and $\mc{A}^*\n\mc{X}=\mc{X}\n\mc{A}^*$, then $\mc{A}\n\mc{A}^{(1,3)}\n\mc{X}=\mc{X}\n\mc{A}\n\mc{A}^{(1,3)}$.
\end{lemma}
\begin{proof}
Let $\mc{A}\n\mc{X}=\mc{X}\n\mc{A}$ and $\mc{A}^*\n\mc{X}=\mc{X}\n\mc{A}^*$. Now using these conditions, we obtain
\begin{equation}\label{eq3}
  \mc{X}\n\mc{A}\n\mc{A}^{(1,3)}=\mc{A}\n\mc{X}\n\mc{A}^{(1,3)} =\mc{A}\mc{A}^{(1,3)}\n\mc{A}\n\mc{X}\n\mc{A}^{(1,3)}=\mc{A}\mc{A}^{(1,3)}\n\n\mc{X}\n\mc{A}\mc{A}^{(1,3)}
\end{equation}
and
\begin{eqnarray}\label{eq4}
\nonumber
\mc{A}\n\mc{A}^{(1,3)}\n\mc{X}&=&\left(\mc{A}^{(1,3)}\right)^*\n\mc{A}^*\n\mc{X}=\left(\mc{A}^{(1,3)}\right)^*\n\mc{X}\n\mc{A}^*\\
\nonumber
&=&\left(\mc{A}^{(1,3)}\right)^*\n\mc{X}\n\left(\mc{A}\n\mc{A}^{(1,3)}\n\mc{A}\right)^*=\left(\mc{A}^{(1,3)}\right)^*\n\mc{X}\n\mc{A}^*\n\mc{A}\n\mc{A}^{(1,3)}\\
&=&\left(\mc{A}^{(1,3)}\right)^*\n\mc{A}^*\n\mc{X}\n\mc{A}\n\mc{A}^{(1,3)}=\mc{A}\n\mc{A}^{(1,3)}\n\mc{X}\n\mc{A}\n\mc{A}^{(1,3)}.
\end{eqnarray}
From Eqs \eqref{eq3} and \eqref{eq4}, we get $\mc{A}\n\mc{A}^{(1,3)}\n\mc{X}=\mc{X}\n\mc{A}\n\mc{A}^{(1,3)}$.
\end{proof}
In view of Lemma \ref{lm2.5}, \ref{lm2.12}, and \ref{lm2.13}, we obtain the following result as a corollary.
\begin{corollary}\label{cor2.14}
Let $\mc{A}\in \mathbb{C}^{\textbf{N}(n)\times \textbf{N}(n)}$ and ind$(\mc{A})=1$. If there exist a tensor $\mc{X}\in \mathbb{C}^{\textbf{N}(n)\times \textbf{N}(n)}$ such that $\mc{A}\n\mc{X}=\mc{X}\n\mc{A}$ and $\mc{A}^*\n\mc{X}=\mc{X}\n\mc{A}^*$, then $\mc{A}^{\core}\n\mc{X}=\mc{X}\n\mc{A}^{\core}$.
\end{corollary}

\section{Reverse-order law}
In this section, we discuss various necessary and sufficient conditions of the reverse-order law for the core inverses of tensors. The very first result obtained below deals with sufficient condition for the reverse-order law of tensors 
\begin{theorem}\label{thm3.1}
Let $\mc{A},~\mc{B},~\mc{A}\n\mc{B}\in\C^{\textbf{N}(n)\times \textbf{N}(n)}$ be core tensors. If $\mc{A}\n\mc{B}\n\mc{B}^{\core}=\mc{B}\n\mc{B}^{\core}\n\mc{A}$ and $\mc{B}\n\mc{A}\n\mc{A}^{\core}=\mc{A}\n\mc{A}^{\core}\n\mc{B}$, then
\begin{equation*}
(\mc{A}\n\mc{B})^{\core}=\mc{B}^{\core}\n\mc{A}^{\core}.
\end{equation*}
\end{theorem}

\begin{proof}
Let $\mc{A}\n\mc{B}\n\mc{B}^{\core}=\mc{B}\n\mc{B}^{\core}\n\mc{A}$ and $\mc{B}\n\mc{A}\n\mc{A}^{\core}=\mc{A}\n\mc{A}^{\core}\n\mc{B}$. By taking both sides, transpose of conjugate  and using core definition, we obtain
\begin{equation}\label{eq5}
  \mc{B}\n\mc{B}^{\core}\n\mc{A}^*= \mc{A}^*\n\mc{B}\n\mc{B}^{\core} \mbox{ and } \mc{A}\n\mc{A}^{\core}\n\mc{B}^*=\mc{B}^*\n\mc{A}\n\mc{A}^{\core}.
\end{equation}
So by Corollary \ref{cor2.14}, we have
\begin{equation}\label{eq6}
  \mc{A}^{\core}\n\mc{B}\n\mc{B}^{\core}=\mc{B}\n\mc{B}^{\core}\n\mc{A}^{\core} \mbox{ and }  \mc{B}^{\core}\n\mc{A}\n\mc{A}^{\core}=\mc{A}\n\mc{A}^{\core}\n\mc{B}^{\core}.
\end{equation}
Let $\mc{X}=\mc{B}^{\core}\mc{A}^{\core}$. Now applying Eqs \eqref{eq5}, \eqref{eq6}, we get
\begin{eqnarray*}
\mc{A}\n\mc{B}\n\mc{X}\n\mc{A}\n\mc{B}&=&\mc{A}\n(\mc{B}\n\mc{B}^{\core}\n\mc{A}^{\core})\n\mc{A}\n\mc{B}=\mc{A}\n\mc{A}^{\core}\n(\mc{B}\n\mc{B}^{\core}\n\mc{A})\n\mc{B}\\
&=&\mc{A}\n\mc{A}^{\core}\n\mc{A}\n\mc{B}\n\mc{B}^{\core}\n\mc{B}=\mc{A}\n\mc{B},
\end{eqnarray*}
\begin{eqnarray*}
\mc{X}&=&\mc{B}^{\core}\n\mc{A}^{\core}=\mc{B}^{\core}\n\mc{A}\n(\mc{A}^{\core})^2=\mc{A}\n\mc{A}^{\core}\n\mc{B}^{\core}\n\mc{A}^{\core}=\mc{A}\n\mc{A}^{\core}\n\mc{B}\n(\mc{B}^{\core})^2\n\mc{A}^{\core}\\
&=& \mc{A}\n(\mc{A}^{\core}\n\mc{B}\n\mc{B}^{\core})\n\mc{B}^{\core}\n\mc{A}^{\core}=\mc{A}\n\mc{B}\n\mc{B}^{\core}\n\mc{A}^{\core}\n\mc{B}^{\core}\n\mc{A}^{\core}=\mc{A}\n\mc{B}\n\mc{X}^2,
\end{eqnarray*}
and
\begin{eqnarray*}
(\mc{A}\n\mc{B}\n\mc{X})^*&=&(\mc{A}\n\mc{B}\n\mc{B}^{\core}\n\mc{A}^{\core})^*=(\mc{A}^{\core})^*\n(\mc{B}^{\core})^*\n\mc{B}^*\n\mc{A}^*\\
&=&(\mc{A}^{\core})^*\n(\mc{B}\n\mc{B}^{\core}\n\mc{A}^*)=(\mc{A}^{\core})^*\n\mc{A}^*\n\mc{B}\n\mc{B}^{\core}=\mc{A}\n(\mc{A}^{\core}\n\mc{B}\n\mc{B}^{\core})\\
&=& \mc{A}\n\mc{B}\n\mc{B}^{\core}\n\mc{A}^{\core}=\mc{A}\n\mc{B}\n\mc{X}.
\end{eqnarray*}
From Lemma \ref{lm2.11} $(a)$, $\mc{X}$ is the core inverse of $\mc{A}\n\mc{B}$. Therefore, $(\mc{A}\n\mc{B})^{\core}=\mc{X}=\mc{B}^{\core}\n\mc{A}^{\core}$.
\end{proof}
If we consider $\mc{B}=\mc{A}^{\core}$ in the above theorem we get the following result. 
\begin{corollary}
Let $\mc{A}\in\C^{\textbf{N}(n)\times \textbf{N}(n)}$ be a core tensors. Then
     $(\mc{A}\n\mc{A}^{\core})^{\core}=\mc{A}\n\mc{A}^{\core}$ and $(\mc{A}^{\core}\n\mc{A})^{\core}=\mc{A}^{\core}\n\mc{A}$.
\end{corollary}

\begin{proposition}
Let $\mc{A}\in\C^{\textbf{N}(n)\times \textbf{N}(n)}$ be a core tensor. If $\mc{A}\n\mc{A}^* = \mc{A}^*\n\mc{A}$, then
$\mc{A}\n\mc{A}^{\core}=\mc{A}^{\core}\n\mc{A}$.
\end{proposition}

Note that, if $\mc{A}\n\mc{B}=\mc{B}\n\mc{A}$ and $\mc{A}^*\n\mc{B}=\mc{B}\n\mc{A}^*$, then from Corollary  \ref{cor2.14}, we can easily get $\mc{A}\n\mc{B}\n\mc{B}^{\core}=\mc{B}\n\mc{B}^{\core}\n\mc{A}$ and $\mc{B}\n\mc{A}\n\mc{A}^{\core}=\mc{A}\n\mc{A}^{\core}\n\mc{B}$. In conclusion, we state the following as a corollary.

\begin{corollary}\label{cor3.2}
Let $\mc{A},~\mc{B}\in\C^{\textbf{N}(n)\times \textbf{N}(n)}$ be core tensors. If  $\mc{A}\n\mc{B}=\mc{B}\n\mc{A}$ and $\mc{A}^*\n\mc{B}=\mc{B}\n\mc{A}^*$, then
\begin{equation*}
(\mc{A}\n\mc{B})^{\core}=\mc{B}^{\core}\n\mc{A}^{\core}.
\end{equation*}
\end{corollary}

Converse of the above result is not true, and is shown in the following example.
\begin{example}
Let
 $\mc{A} = (a_{ijkl})  \in \mathbb{R}^{2\times 3\times 2\times 3}$ and $\mc{B}
 = (b_{ijkl})\in \mathbb{R}^{2\times 3\times 2\times 3}, ({1 \leq i,k \leq 2}$ and $ {1 \leq j,l \leq 3})$ be two core tensors such that
\begin{eqnarray*}
a_{ij11} =
    \begin{pmatrix}
    1 & 0 & 1\\
    0 & 0 & -1\\
    \end{pmatrix},
a_{ij12} =
    \begin{pmatrix}
    0 & 0 & 0\\
    0 & 0 & 0\\
    \end{pmatrix},
a_{ij13} =
    \begin{pmatrix}
    1 & 0 & 0\\
    1 & 0 & 0\\
    \end{pmatrix},
    \end{eqnarray*}
and
\begin{eqnarray*}
a_{ij21} =
\begin{pmatrix}
    0 & 0 & 1\\
    0 & 1 & 0\\
\end{pmatrix}, ~
    a_{ij22} =
\begin{pmatrix}
    0 & 0 & 0\\
    1 & 1 & 0\\
\end{pmatrix}, ~
    a_{ij23} =
\begin{pmatrix}
    -1 & 0 & 0\\
    0 & 0 & 0\\
\end{pmatrix}. ~
\end{eqnarray*}
\begin{eqnarray*}
b_{ij11} =
    \begin{pmatrix}
    1 & 0 & 0\\
    0 & 0 & 0\\
    \end{pmatrix},
b_{ij12} =
    \begin{pmatrix}
    0 & 0 & 0\\
    0 & 0 & 0\\
    \end{pmatrix},
b_{ij13} =
    \begin{pmatrix}
    0 & 0 & 1\\
    0 & 0 & 0\\
    \end{pmatrix},
    \end{eqnarray*}
and
\begin{eqnarray*}
b_{ij21} =
\begin{pmatrix}
    0 & 0 & 0\\
    1 & 0 & 0\\
\end{pmatrix}, ~
    b_{ij22} =
\begin{pmatrix}
    0 & 0 & 0\\
    0 & 1 & 0\\
\end{pmatrix}, ~
    b_{ij23} =
\begin{pmatrix}
    0 & 0 & 0\\
    0 & 0 & 1\\
\end{pmatrix}, ~
\end{eqnarray*}
respectively.  Then   ${\mc{A}}^{\core} = (x_{ijkl}) \in \mathbb{R}^{2\times 3\times 2\times 3}$ and ${\mc{B}}^{\core} = (y_{ijkl})   \in \mathbb{R}^{2\times 3\times 2\times 3}$, where
\begin{eqnarray*}
x_{ij11} =
    \begin{pmatrix}
    0 & 0 & 0\\
    0 & 0 & -1\\
    \end{pmatrix},
x_{ij12} =
    \begin{pmatrix}
    0 & 0 & 0\\
    0 & 0 & 0\\
    \end{pmatrix},
x_{ij13} =
    \begin{pmatrix}
    0 & 0 & 1\\
    1 & -1 & 1\\
    \end{pmatrix},
    \end{eqnarray*}
and
\begin{eqnarray*}
x_{ij21} =
\begin{pmatrix}
    0 & 0 & 1\\
    0 & 0 & 1\\
\end{pmatrix}, ~
x_{ij22} =
\begin{pmatrix}
    0 & 0 & -1\\
    0 & 1 & -1\\
\end{pmatrix}, ~
    x_{ij23} =
\begin{pmatrix}
    -1 & 0 & 1\\
    1 & -1 & 0\\
\end{pmatrix}. ~
\end{eqnarray*}
\begin{eqnarray*}
y_{ij11} =
    \begin{pmatrix}
    1 & 0 & 0\\
    0 & 0 & 0\\
    \end{pmatrix},
y_{ij12} =
    \begin{pmatrix}
    0 & 0 & 0\\
    0 & 0 & 0\\
    \end{pmatrix},
y_{ij13} =
    \begin{pmatrix}
    0 & 0 & 1\\
    0 & 0 & 0\\
    \end{pmatrix},
    \end{eqnarray*}
and
\begin{eqnarray*}
y_{ij21} =
\begin{pmatrix}
    0 & 0 & 0\\
    1 & 0 & 0\\
\end{pmatrix}, ~
    y_{ij22} =
\begin{pmatrix}
    0 & 0 & 0\\
    -8 & 1 & 0\\
\end{pmatrix}, ~
    y_{ij23} =
\begin{pmatrix}
    0 & 0 & 0\\
    0 & 0 & 1\\
\end{pmatrix}, ~
\end{eqnarray*}
respectively. So  $\mc{B}^{\core} *_N \mc{A}^{\core}  = {(\mc{A}*_N\mc{B})}^{\core}=
 (d_{ijkl})\in \mathbb{R}^{2\times  \times 2\times 3}$,
 where
\begin{eqnarray*}
d_{ij11} =
    \begin{pmatrix}
    0 & 0 & 0\\
    0 & 0 & -1\\
    \end{pmatrix},
d_{ij12} =
    \begin{pmatrix}
    0 & 0 & 0\\
    0 & 0 & 0\\
    \end{pmatrix},
d_{ij13} =
    \begin{pmatrix}
    0 & 0 & 1\\
    9 & -1 & 1\\
    \end{pmatrix},
    \end{eqnarray*}
and
\begin{eqnarray*}
d_{ij21} =
\begin{pmatrix}
    0 & 0 & 1\\
    0 & 0 & 1\\
\end{pmatrix}, ~
    d_{ij22} =
\begin{pmatrix}
    0 & 0 & -1\\
    -8 & 1 & -1\\
\end{pmatrix}, ~
    d_{ij23} =
\begin{pmatrix}
    -1 & 0 & 1\\
    9 & -1 & 0\\
\end{pmatrix}.
\end{eqnarray*}
 However  $(\mc{A}*_N\mc{B}) \neq (\mc{B}*_N\mc{A})$ and  $\mc{A}^*\n\mc{B}\neq \mc{B}\n\mc{A}^*$.
\end{example}

 \begin{theorem}\label{thm3.3}
 Let $\mc{A},~\mc{B},~\mc{A}\n\mc{B}\in\C^{\textbf{N}(n)\times \textbf{N}(n)}$ be core tensors. Then
\begin{equation*}
\rg(\mc{A}\n\mc{B}^{\core})\subseteq\rg(\mc{B}^{\core}\n\mc{A}^{\core}) ~~and~~ \mc{B}^*\n\mc{A}^{\#}(\mc{I}-(\mc{A}\n\mc{B}\n\mc{B}^{\core}\n\mc{A}^{\core})^*)\n\mc{A}=\mc{O}
\end{equation*}
if and only if
 \begin{equation*}
(\mc{A}\n\mc{B})^{\core}=\mc{B}^{\core}\n\mc{A}^{\core}
 \end{equation*}
 \end{theorem}

 \begin{proof}
 Let $(\mc{A}\n\mc{B})^{\core}=\mc{B}^{\core}\n\mc{A}^{\core}$. Then $\rg(\mc{A}\n\mc{B}^{\core})\subseteq\rg(\mc{B}^{\core}\n\mc{A}^{\core})$ since  $\mc{A}\n\mc{B}^{\core}=\mc{A}\n\mc{B}\n(\mc{B}^{\core})^2=(\mc{A}\n\mc{B})^{\core}\n(\mc{A}\n\mc{B})^2\n(\mc{B}^{\core})^2=\mc{B}^{\core}\n\mc{A}^{\core}\n(\mc{A}\n\mc{B})^2\n(\mc{B}^{\core})^2$. Further,
 \begin{eqnarray*}
  \mc{B}^*\n\mc{A}^{\#}\n\mc{A}&=& \mc{B}^*\n\mc{A}\n\mc{A}^{\#}=\mc{B}^*\n\mc{B}\n\mc{B}^{\core}\n\mc{A}\n\mc{A}^{\#}=\mc{B}^*\n\mc{B}\n\mc{B}^{\core}\n\mc{A}^{\core}\n\mc{A}^2\n\mc{A}^{\#}\\
  &=&\mc{B}^*\n\mc{B}\n\mc{B}^{\core}\n\mc{A}^{\core}\n\mc{A}=\mc{B}^*\n\mc{B}\n\mc{B}^{\core}\n\mc{A}^{\core}\n\mc{A}\n\mc{B}\n\mc{B}^{\core}\n\mc{A}^{\core}\n\mc{A}\\
  &=& \mc{B}^*\n\mc{A}^{\core}\n\mc{A}\n\mc{B}\n\mc{B}^{\core}\n\mc{A}^{\core}\n\mc{A}=\mc{B}^*\n\mc{A}^{\#}\n\mc{A}\n\mc{B}\n\mc{B}^{\core}\n\mc{A}^{\core}\n\mc{A}\\
  &=& \mc{B}^*\n\mc{A}^{\#}\n(\mc{A}\n\mc{B}\n\mc{B}^{\core}\n\mc{A}^{\core})^*\n\mc{A}.
 \end{eqnarray*}
Thus $\mc{B}^*\n\mc{A}^{\#}(\mc{I}-(\mc{A}\n\mc{B}\n\mc{B}^{\core}\n\mc{A}^{\core})^*)\n\mc{A}=\mc{O}$.\\
Conversely, let $\mc{B}^*\n\mc{A}^{\#}\n\mc{A}=\mc{B}^*\n\mc{A}^{\#}\n(\mc{A}\n\mc{B}\n\mc{B}^{\core}\n\mc{A}^{\core})^*\n\mc{A}$. Then
\begin{eqnarray*}
\mc{B}^*\n\mc{A}^{\core}&=&\mc{B}^*\n\mc{A}^{\#}\n\mc{A}\n\mc{A}^{\dagger}=\mc{B}^*\n\mc{A}^{\#}\n(\mc{A}\n\mc{B}\n\mc{B}^{\core}\n\mc{A}^{\core})^*\n\mc{A}\n\mc{A}^\dagger\\
&=& \mc{B}^*\n\mc{A}^{\#}\n(\mc{A}^{\core})^*\n\mc{B}\n\mc{B}^{\core}\n\mc{A}^{*}\mc{A}\n\mc{A}^\dagger= \mc{B}^*\n\mc{A}^{\core}\n(\mc{A}^{\core})^*\n\mc{B}\n\mc{B}^{\core}\n\mc{A}^{*}.
\end{eqnarray*}
This yields, $(\mc{A}^{\core})^*\n\mc{B}=\mc{A}\n\mc{B}\n\mc{B}^{\core}\n\mc{A}^{\core}\n(\mc{A}^{\core})^*\n\mc{B}$. Using this relation, we obtain
\begin{eqnarray}\label{eq7}
\nonumber
(\mc{B}^{\core}\n\mc{A}^{\core})^*&=&(\mc{A}^{\core})^*\n(\mc{B}^{\core})^*=(\mc{A}^{\core})^*\mc{B}\n\mc{B}^{\core}\n(\mc{B}^{\core})^*\\
\nonumber
&=&\mc{A}\n\mc{B}\n\mc{B}^{\core}\n\mc{A}^{\core}\n(\mc{A}^{\core})^*\n\mc{B}\n\mc{B}^{\core}\n(\mc{B}^{\core})^*\\
\nonumber
&=&\mc{A}\n\mc{B}\n\mc{B}^{\core}\n\mc{A}^{\core}\n(\mc{A}^{\core})^*\n(\mc{B}^{\core})^*\\
&=&\mc{A}\n\mc{B}\n\mc{B}^{\core}\n\mc{A}^{\core}\n(\mc{B}^{\core}\n\mc{A}^{\core})^*.
\end{eqnarray}
Which implies $\mc{B}^{\core}\n\mc{A}^{\core}=\mc{B}^{\core}\n\mc{A}^{\core}\n(\mc{A}\n\mc{B}\n\mc{B}^{\core}\n\mc{A}^{\core})^*$. Therefore ,
\begin{equation}\label{eq8}
 \mc{A}\n\mc{B}\n\mc{B}^{\core}\n\mc{A}^{\core}=  \mc{A}\n\mc{B}\n\mc{B}^{\core}\n\mc{A}^{\core}\n(\mc{A}\n\mc{B}\n\mc{B}^{\core}\n\mc{A}^{\core})^*.
\end{equation}
Applying Eq. \eqref{eq7} and Eq. \eqref{eq8}, we get
\begin{equation}\label{eq9}
\mc{A}\n\mc{B}\n\mc{B}^{\core}\n\mc{A}^{\core}=(\mc{B}^{\core}\n\mc{A}^{\core})^*\n(\mc{A}\n\mc{B})^*=(\mc{A}\n\mc{B}\n\mc{B}^{\core}\n\mc{A}^{\core})^*.
\end{equation}
Further, from Eq. \eqref{eq7} and Eq. \eqref{eq9}, we obtain $\mc{B}^{\core}\n\mc{A}^{\core}=\mc{B}^{\core}\n\mc{A}^{\core}\n\mc{A}\n\mc{B}\n\mc{B}^{\core}\n\mc{A}^{\core}$.

Next we will claim $\mc{B}^{\core}\n\mc{A}^{\core}\n(\mc{A}\n\mc{B})^2=\mc{B}^{\core}\n\mc{A}^{\core}$. From
$\rg(\mc{A}\n\mc{B}^{\core})\subseteq\rg(\mc{B}^{\core}\n\mc{A}^{\core})$, we have $\mc{A}\n\mc{B}^{\core}=\mc{B}^{\core}\n\mc{A}^{\core}\n\mc{U}$ for some $\mc{U}\in\mathbb{C}^{\textbf{N}(n)\times \textbf{N}(n)}$. Now by Eq. \eqref{eq8},
\begin{eqnarray*}
\mc{B}^{\core}\n\mc{A}^{\core}\n(\mc{A}\n\mc{B})^2&=& \mc{B}^{\core}\n\mc{A}^{\core}\n\mc{A}\n\mc{B}\n\mc{A}\n\mc{B}^{\core}\n\mc{B}^2\\
&=&\mc{B}^{\core}\n\mc{A}^{\core}\n\mc{A}\n\mc{B}\n\mc{B}^{\core}\n\mc{A}^{\core}\n\mc{U}\n\mc{B}^2=\mc{B}^{\core}\n\mc{A}^{\core}\n\mc{U}\n\mc{B}^2\\
&=&\mc{A}\n\mc{B}^{\core}\n\mc{B}^2=\mc{A}\n\mc{B}.
\end{eqnarray*}
So by Lemma \ref{lm2.11}, $\mc{B}^{\core}\n\mc{A}^{\core}$ is the core inverse of $\mc{A}\n\mc{B}$.
 \end{proof}

 A few necessary conditions of reverse-order law is presented in the following theorem.
  \begin{theorem}\label{thm3.4}
 Let $\mc{A},~\mc{B}\in\C^{\textbf{N}(n)\times \textbf{N}(n)}$ be core tensors. If $(\mc{A}\n\mc{B})^{\core}=\mc{B}^{\core}\n\mc{A}^{\core}$, then the following holds:
 \begin{enumerate}
     \item[(a)] $\rg(\mc{A}\n\mc{B})\subseteq\rg(\mc{B}\n\mc{A})$;
     \item[(b)] $\mc{B}\n\mc{B}^{\core}\n\mc{A}^{\core}\in \mc{C}\{3,6\}$, where $\mc{C}=\mc{A}\n\mc{B}\n\mc{B}^{\core}$
 \end{enumerate}
 \end{theorem}
 \begin{proof}
 $(a)$ Let $(\mc{A}\n\mc{B})^{\core}=\mc{B}^{\core}\n\mc{A}^{\core}$. Now
 \begin{eqnarray}\label{eqar}
 \mc{A}\n\mc{B}&=&\mc{B}^{\core}\n\mc{A}^{\core}\n(\mc{A}\n\mc{B})^2=\mc{B}\n(\mc{B}^{\core})^2\n\mc{A}^{\core}\n(\mc{A}\n\mc{B})^2=\mc{B}\n\mc{B}^{\core}\n(\mc{A}\n\mc{B})\\
 \nonumber
 &=& \mc{B}\n\mc{B}^{\core}\n\mc{A}^{\core}\n\mc{A}^2\n\mc{B}=\mc{B}\n(\mc{A}\n\mc{B})^{\core}\n\mc{A}^2\n\mc{B}\\
 \nonumber
 &=& \mc{B}\n\mc{A}\n\mc{B}\n\left((\mc{A}\n\mc{B})^{\core}\right)^2\n\mc{A}^2\n\mc{B}.
 \end{eqnarray}
 This yields $\rg(\mc{A}\n\mc{B})\subseteq\rg(\mc{B}\n\mc{A})$.\\
 $(b)$ As $\mc{C}\n\mc{B}\n\mc{B}^{\core}\n\mc{A}^{\core}=\mc{A}\n\mc{B}\n\mc{B}^{\core}\n\mc{B}\n\mc{B}^{\core}\n\mc{A}^{\core}=\mc{A}\n\mc{B}\n\mc{B}^{\core}\n\mc{A}^{\core}=\mc{A}\n\mc{B}\n(\mc{A}\n\mc{B})^{\core}$. So we have $(\mc{C}\n\mc{B}\n\mc{B}^{\core}\n\mc{A}^{\core})^*=(\mc{A}\n\mc{B}\n(\mc{A}\n\mc{B})^{\core})^*=\mc{A}\n\mc{B}\n(\mc{A}\n\mc{B})^{\core}=\mc{C}\n\mc{B}\n\mc{B}^{\core}\n\mc{A}^{\core}$. Thus $\mc{B}\n\mc{B}^{\core}\n\mc{A}^{\core}\in\mc{C}\{3\}$. By using $\mc{A}\n\mc{B}=\mc{B}\n\mc{B}^{\core}\n(\mc{A}\n\mc{B})$ from part $(a)$, we obtain
 \begin{eqnarray*}
 \mc{B}\n\mc{B}^{\core}\n\mc{A}^{\core}\n\mc{C}^2&=& \mc{B}\n\mc{B}^{\core}\n\mc{A}^{\core}\n\mc{A}\n\mc{B}\n\mc{B}^{\core}\n\mc{A}^{\core}\n\mc{A}^2\n\mc{B}\n\mc{B}^{\core}\\
 &=&\mc{B}\n\mc{B}^{\core}\n\mc{A}^{\core}\n\mc{A}^2\n\mc{B}\n\mc{B}^{\core}=\mc{B}\n\mc{B}^{\core}\n\mc{A}\n\mc{B}\n\mc{B}^{\core}=\mc{A}\n\mc{B}\n\mc{B}^{\core}.
 \end{eqnarray*}
 Hence $\mc{B}\n\mc{B}^{\core}\n\mc{A}^{\core}\in\mc{C}\{6\}$ and completes the proof.
 \end{proof}
\begin{theorem}\label{thm3.5}
Let $\mc{A}$ and $\mc{B}\in\C^{\textbf{N}(n)\times \textbf{N}(n)}$ be two core tensors. If $\mc{A}^2=\mc{B}\n\mc{A}$, then
\begin{enumerate}
    \item[(a)] $ind(\mc{A}\n\mc{B})=1$ and $(\mc{A}\n\mc{B})^{\core}=\mc{B}^{\core}\n\mc{A}^{\core}$;
    \item[(b)] $ind(\mc{A}\n\mc{B}\n\mc{B}^{\core})=1$ and $(\mc{A}\n\mc{B}\n\mc{B}^{\core})^{\core}=\mc{B}\n\mc{B}^{\core}\n\mc{A}^{\core}$.
\end{enumerate}
\end{theorem}
\begin{proof}
$(a)$ It is trivial that $\rg((\mc{A}\n\mc{B})^2)\subseteq\rg(\mc{A}\n\mc{B})$. Since $\mc{A}\n\mc{B}=\mc{A}^5\n(\mc{A}^{\#})^4\n\mc{B}=\mc{A}\n(\mc{B}\n\mc{A})^2\n(\mc{A}^{\#})^4\n\mc{B}=(\mc{A}\n\mc{B})^2\n(\mc{A}^{\#})^3\n\mc{B}$. It follows that $\rg(\mc{A}\n\mc{B})\subseteq\rg((\mc{A}\n\mc{B})^2)$ and hence $ind(\mc{A}\n\mc{B})=1$. From the condition $\mc{A}^2=\mc{B}\n\mc{A}$, we have
\begin{equation}\label{eq11}
\mc{A}=\mc{A}^2\n\mc{A}^{\#}=\mc{B}\n\mc{A}\n\mc{A}^{\#}=\mc{B}^{\core}\n\mc{B}^2\n\mc{A}\n\mc{A}^{\#}=\mc{B}^{\core}\n\mc{B}\n\mc{A}^2\n\mc{A}^{\#}= \mc{B}^{\core}\n\mc{A}^2,
\end{equation}
and
$\mc{A}\n\mc{A}^{\core}=\mc{A}^2\n(\mc{A}^{\core})^2=\mc{B}\n\mc{A}\n(\mc{A}^{\core})^2=\mc{B}\n\mc{A}^{\core}$. From Eq. \eqref{eq11}, we get
\begin{equation}\label{eq12}
 \mc{A}^{\core}=\mc{A}\n( \mc{A}^{\core})^2=\mc{B}^{\core}\n\mc{A}^2\n( \mc{A}^{\core})^2=\mc{B}^{\core}\n\mc{B}\n\mc{A}\n( \mc{A}^{\core})^2=\mc{B}^{\core}\n\mc{B}\n\mc{A}^{\core},
\end{equation}
\begin{equation*}
 \mbox{and } \mc{B}^{\core}\n\mc{A}^{\core}\n(\mc{A}\n\mc{B})^2= \mc{B}^{\core}\n\mc{A}^{\core}\n\mc{A}^3\n\mc{B}=\mc{B}^{\core}\n\mc{A}^2\n\mc{B}=\mc{A}\n\mc{B}.
\end{equation*}
Further from Eq. \eqref{eq11} and Eq. \eqref{eq12}, we have
\begin{eqnarray}\label{eq13}
\nonumber
\mc{A}\n\mc{B}\n\mc{B}^{\core}\n\mc{A}^{\core}&=&\mc{A}\n\mc{B}\n\mc{B}^{\core}\n\mc{B}^{\core}\n\mc{B}\n\mc{A}^{\core}=\mc{A}\n\mc{B}^{\core}\n\mc{B}\n\mc{A}^{\core}=\mc{A}\n\mc{B}^{\core}\n\mc{A}\n\mc{A}^{\core}\\
&=&\mc{A}\n\mc{B}^{\core}\n\mc{A}^2\n(\mc{A}^{\core})^2=\mc{A}^2\n(\mc{A}^{\core})^2=\mc{A}\n\mc{A}^{\core}.
\end{eqnarray}
Applying Eq. \eqref{eq13}, we obtain
\begin{equation*}
    (\mc{A}\n\mc{B}\n\mc{B}^{\core}\n\mc{A}^{\core})^*=(\mc{A}\n\mc{A}^{\core})^*=\mc{A}\n\mc{A}^{\core}=\mc{A}\n\mc{B}\n\mc{B}^{\core}\n\mc{A}^{\core},
\end{equation*}
and $ \mc{B}^{\core}\n\mc{A}^{\core}\n\mc{A}\n\mc{B}\n\mc{B}^{\core}\n\mc{A}^{\core} =\mc{B}^{\core}\n\mc{A}^{\core}\n\mc{A}\n\mc{A}^{\core}= \mc{B}^{\core}\n\mc{A}^{\core}$. Therefore, $\mc{B}^{\core}\n\mc{A}^{\core}$ is the core inverse of $\mc{A}\n\mc{B}$.\\
$(b)$ Clearly $\rg((\mc{A}\n\mc{B}\n\mc{B}^{\core})^2)\subseteq\rg(\mc{A}\n\mc{B}\n\mc{B}^{\core})$. From $\rg(\mc{A}\n\mc{B})=\rg((\mc{A}\n\mc{B})^2)$, we have $\mc{A}\n\mc{B}=(\mc{A}\n\mc{B})^\n\mc{U}$ for some $\mc{U}\in\C^{\textbf{N}(n)\times \textbf{N}(n)}$. Now
\begin{eqnarray*}
  \mc{A}\n\mc{B}\n\mc{B}^{\core} &=& \mc{A}\n\mc{B}\n\mc{A}\n\mc{B}\n\mc{U}\n\mc{B}^{\core}=\mc{A}\n\mc{B}\n\mc{B}^{\core}\n\mc{A}^{\core}\n(\mc{A}\n\mc{B})^2\n\mc{U}\n\mc{B}^{\core} \\
   &=& \mc{A}\n\mc{B}\n\mc{B}^{\core}\n\mc{A}^{\core}\n\mc{A}^3\n\mc{B}\n\mc{U}\n\mc{B}^{\core}=\mc{A}\n\mc{B}\n\mc{B}^{\core}\n\mc{A}^{2}\n\mc{B}\n\mc{U}\n\mc{B}^{\core}\\
   &=&\mc{A}\n\mc{B}\n\mc{B}^{\core}\n\mc{A}\n\mc{B}^{\core}\n\mc{A}^{2}\n\mc{B}\n\mc{U}\n\mc{B}^{\core}\\
   &=&\mc{A}\n\mc{B}\n\mc{B}^{\core}\n\mc{A}\n\mc{B}\n(\mc{B}^{\core})^2\n\mc{A}^{2}\n\mc{B}\n\mc{U}\n\mc{B}^{\core}
   \\
   &=&(\mc{A}\n\mc{B}\n\mc{B}^{\core})^2\n\mc{Z},\mbox{ where }\mc{Z}=\mc{B}^{\core}\n\mc{A}^{2}\n\mc{B}\n\mc{U}\n\mc{B}^{\core}.
\end{eqnarray*}
Thus $\rg(\mc{A}\n\mc{B}\n\mc{B}^{\core})\subseteq \rg((\mc{A}\n\mc{B}\n\mc{B}^{\core})^2)$ and hence $ind(\mc{A}\n\mc{B}\n\mc{B}^{\core})=1$. In view of part $(a)$ and Theorem \ref{thm3.4}, it is sufficient to show only $\mc{B}\n\mc{B}^{\core}\n\mc{A}^{\core}\in \mc{C}\{7\}$, where $\mc{C}=\mc{A}\n\mc{B}\n\mc{B}^{\core}$.  By using Eq. \eqref{eqar}, we obtain
\begin{eqnarray*}
\mc{B}\n\mc{B}^{\core}\n\mc{A}^{\core}\n\mc{C}^2&=&\mc{B}\n\mc{B}^{\core}\n\mc{A}^{\core}\n\mc{A}\n\mc{B}\n\mc{B}^{\core}\n\mc{A}\n\mc{B}\n\mc{B}^{\core}\\
&=&\mc{B}\n(\mc{A}\n\mc{B})^{\core}\n\mc{A}\n\mc{B}\n(\mc{A}\n\mc{B})^{\core}\n\mc{A}^2\n\mc{B}\n\mc{B}^{\core}\\
&=& \mc{B}\n\mc{B}^{\core}\n\mc{A}^{\core}\n\mc{A}^2\n\mc{B}\n\mc{B}^{\core}=\mc{B}\n\mc{B}^{\core}\n\mc{A}\n\mc{B}\n\mc{B}^{\core}\\
&=&\mc{A}\n\mc{B}\n\mc{B}^{\core}=\mc{C}.
\end{eqnarray*}
Thus $\mc{B}\n\mc{B}^{\core}\n\mc{A}^{\core}\in \mc{C}\{7\}$.
\end{proof}
\begin{theorem}\label{thm3.6}
Let $\mc{A},~\mc{B}, \mc{A}\n\mc{B}\in\C^{\textbf{N}(n)\times \textbf{N}(n)}$   be core tensors and  $\rg(\mc{A}^*\n\mc{B})=\rg(\mc{B}\n\mc{A}^*)$. Then   $(\mc{A}\n\mc{B})^{\core}=\mc{B}^{\core}\n\mc{A}^{\core}$ if and only if the following holds
\begin{enumerate}
    \item[(a)] $\rg(\mc{B}^{\core}\n\mc{A})\subseteq\rg(\mc{A}\n\mc{B})\subseteq\rg(\mc{B}\n\mc{A})$.
    \item[(b)] $\mc{A}\n\mc{A}^{\core}$ and $\mc{B}\n\mc{B}^{\core}$ are commute. 
\end{enumerate}
\end{theorem}
\begin{proof}
Let $(\mc{A}\n\mc{B})^{\core}=\mc{B}^{\core}\n\mc{A}^{\core}$. By Theorem \ref{thm3.4}, we have $\rg(\mc{A}\n\mc{B})\subseteq\rg(\mc{B}\n\mc{A})$. Now
\begin{eqnarray*}
\mc{B}^{\core}\n\mc{A}=\mc{B}^{\core}\n\mc{A}^{\core}\n\mc{A}^2=(\mc{A}\n\mc{B})^{\core}\n\mc{A}^2=\mc{A}\n\mc{B}\n((\mc{A}\n\mc{B})^{\core})^2\n\mc{A}^2.
\end{eqnarray*}
This implies $\rg(\mc{B}^{\core}\n\mc{A})\subseteq\rg(\mc{A}\n\mc{B})$. Since $\rg(\mc{A}^*\n\mc{B})=\rg(\mc{B}\n\mc{A}^*)$, it gives $\mc{B}^*\n\mc{A}=\mc{U}\n\mc{A}\n\mc{B}^*$ for some $\mc{U}\in \C^{\textbf{N}(n)\times \textbf{N}(n)}$. Applying this, we get 
\begin{equation}\label{eq14}
 \mc{B}^*\n\mc{A}=\mc{U}\n\mc{A}\n\mc{B}^*=\mc{U}\n\mc{A}\n(\mc{B}\n\mc{B}^{\core}\n\mc{B})^*=\mc{B}^*\n\mc{A}\n\mc{B}\n\mc{B}^{\core}.
\end{equation}
From Eq. \eqref{eqar} and Eq. \eqref{eq14}, we obtain
\begin{eqnarray*}
\mc{B}\n\mc{B}^{\core}\n\mc{A}\n\mc{A}^{\core}&=&(\mc{B}^{\core})^*\n\mc{B}^{*}\n\mc{A}\n\mc{A}^{\core}=(\mc{B}^{\core})^*\n\mc{B}^*\n\mc{A}\n\mc{B}\n\mc{B}^{\core}\n\mc{A}^{\core}\\
&=&\mc{B}\n\mc{B}^{\core}\n\mc{A}\n\mc{B}\n\mc{B}^{\core}\n\mc{A}^{\core}=\mc{A}\n\mc{B}\n\mc{B}^{\core}\n\mc{A}^{\core}\\
&=&(\mc{A}\n\mc{B}\n\mc{B}^{\core}\n\mc{A}^{\core})^*=(\mc{B}\n\mc{B}^{\core}\n\mc{A}\n\mc{A}^{\core})^*=\mc{A}\n\mc{A}^{\core}\n\mc{B}\n\mc{B}^{\core}.
\end{eqnarray*}
Conversely, let $\rg(\mc{A}\n\mc{B})\subseteq\rg(\mc{B}\n\mc{A})$. This yields $\mc{A}\n\mc{B}=\mc{B}\n\mc{A}\n\mc{U}$ for some $\mc{U}$. Now $\mc{A}\n\mc{B}=\mc{B}\n\mc{A}\n\mc{U}=\mc{B}\n\mc{B}^{\core}\n\mc{B}\n\mc{A}\n\mc{U}=\mc{B}\n\mc{B}^{\core}\n\mc{A}\n\mc{B}$. Applying this along with $\rg(\mc{A}^*\n\mc{B})=\rg(\mc{B}\n\mc{A}^*)$, we again obtain $\mc{B}\n\mc{B}^{\core}\n\mc{A}\n\mc{A}^{\core}=\mc{A}\n\mc{B}\n\mc{B}^{\core}\n\mc{A}^{\core}$. By using $\rg(\mc{B}^{\core}\n\mc{A})\subseteq\rg(\mc{A}\n\mc{B})$, we have 
\begin{eqnarray*}
\mc{A}\n\mc{B}\n(\mc{B}^{\core}\n\mc{A}^{\core})^2&=&\mc{B}\n\mc{B}^{\core}\n\mc{A}\n\mc{A}^{\core}\n\mc{B}^{\core}\n\mc{A}^{\core}=\mc{A}\n\mc{A}^{\core}\n\mc{B}\n\mc{B}^{\core}\n\mc{B}^{\core}\n\mc{A}^{\core}\\
&=&\mc{A}\n\mc{A}^{\core}\n(\mc{B}^{\core}\n\mc{A})\n(\mc{A}^{\core})^2=\mc{A}\n\mc{A}^{\core}\n(\mc{A}\n\mc{B}\n\mc{V})\n(\mc{A}^{\core})^2\\
&=&(\mc{A}\n\mc{B}\n\mc{V})\n(\mc{A}^{\core})^2=(\mc{B}^{\core}\n\mc{A})\n(\mc{A}^{\core})^2=\mc{B}^{\core}\n\mc{A}^{\core}.
\end{eqnarray*}
Further, $(\mc{A}\n\mc{B}\n\mc{B}^{\core}\n\mc{A}^{\core})^*=\mc{B}\n\mc{B}^{\core}\n\mc{A}\n\mc{A}^{\core}=\mc{A}\n\mc{B}\n\mc{B}^{\core}\n\mc{A}^{\core}$ and\\ $\mc{A}\n\mc{B}=\mc{B}\n\mc{B}^{\core}\n\mc{A}\n\mc{B}=\mc{B}\n\mc{B}^{\core}\n\mc{A}\n\mc{A}^{\core}\n\mc{A}\n\mc{B}=\mc{A}\n\mc{B}\n\mc{B}^{\core}\n\mc{A}^{\core}\n\mc{A}\n\mc{B}$. By Lemma \ref{lm2.11} $(a)$, $\mc{B}^{\core}\n\mc{A}^{\core}$ is the core inverse of $(\mc{A}\n\mc{B})$.
\end{proof}

On the basis of Theorem \ref{thm3.4} and Theorem \ref{thm3.6}, we state following result as a corollary.
\begin{corollary}\label{cor3.10}
Let $\mc{A},~\mc{B}\in\C^{\textbf{N}(n)\times \textbf{N}(n)}$ be core tensors. If $\rg(\mc{A}^*\n\mc{B})=\rg(\mc{B}\n\mc{A}^*)$ and $(\mc{A}\n\mc{B})^{\core}=\mc{B}^{\core}\n\mc{A}^{\core}$, then $(\mc{A}\n\mc{B}\n\mc{B}^{\core})^{\core}=\mc{B}\n\mc{B}^{\core}\n\mc{A}^{\core}=(\mc{B}\n\mc{B}^{\core})^{\core}\n\mc{A}^{\core}$.
\end{corollary}
\begin{proof}
It is sufficient to show only $\mc{B}\n\mc{B}^{\core}\n\mc{A}^{\core}\in\mc{C}\{7\}$ where $\mc{C}=\mc{A}\n\mc{B}\n\mc{B}^{\core}$. That is $\mc{C}\n(\mc{B}\n\mc{B}^{\core}\n\mc{A}^{\core})^2=\mc{B}\n\mc{B}^{\core}\n\mc{A}^{\core}$. Applying $\mc{A}\n\mc{B}\n\mc{B}^{\core}\n\mc{A}^{\core}=\mc{B}\n\mc{B}^{\core}\n\mc{A}\n\mc{A}^{\core}=\mc{A}\n\mc{A}^{\core}\n\mc{B}\n\mc{B}^{\core}$ of Theorem \ref{thm3.6}, we have 
\begin{eqnarray*}
\mc{C}\n(\mc{B}\n\mc{B}^{\core}\n\mc{A}^{\core})^2&=&\mc{A}\n\mc{B}\n\mc{B}^{\core}\n(\mc{B}\n\mc{B}^{\core}\n\mc{A}^{\core})^2=\mc{A}\n\mc{B}\n\mc{B}^{\core}\n\mc{A}^{\core}\n\mc{B}\n\mc{B}^{\core}\n\mc{A}^{\core}\\
&=&\mc{A}\n\mc{A}^{\core}\n\mc{B}\n\mc{B}^{\core}\n\mc{B}\n\mc{B}^{\core}\n\mc{A}^{\core}=\mc{A}\n\mc{A}^{\core}\n\mc{B}\n\mc{B}^{\core}\n\mc{A}^{\core}\\
&=&\mc{B}\n\mc{B}^{\core}\n\mc{A}\n\mc{A}^{\core}\n\mc{A}^{\core}=\mc{B}\n\mc{B}^{\core}\n\mc{A}^{\core}.
\end{eqnarray*}
\end{proof}
\begin{theorem}\label{thm3.11}
Let $\mc{A},~\mc{B}\in\C^{\textbf{N}(n)\times \textbf{N}(n)}$ be core tensors. Assume that $\rg(\mc{A})\subseteq\rg(\mc{A}\n\mc{B})$ and $\mc{A}$ is EP.  If $(\mc{A}\n\mc{B})^{\core} = \mc{B}^{\core}\n\mc{A}^{\core}$ then $((\mc{A}^{\core})^*\n\mc{B})^{\core}=\mc{B}^{\core}\n\mc{A}^*$.
\end{theorem}
\begin{proof}
Let $(\mc{A}\n\mc{B})^{\core} = \mc{B}^{\core}\n\mc{A}^{\core}$. Since $\mc{A}$ is EP, by Lemma \ref{lm2.6}, $\mc{A}\n\mc{A}^{\core}=\mc{A}^{\core}\n\mc{A}$. This yields $\mc{A}^*=\mc{A}^{\core}\n\mc{A}\n\mc{A}^*$ and $(\mc{A}^{\core}\n\mc{A})^*=\mc{A}^{\core}\n\mc{A}$. Further from the range condition, we have $\mc{A}=\mc{A}\n\mc{B}\n\mc{U}$ for some $\mc{U}\in\C^{\textbf{N}(n)\times \textbf{N}(n)}$. Now
\begin{equation*}
((\mc{A}^{\core})^*\n\mc{B}\n\mc{B}^{\core}\n\mc{A}^*)^*= \mc{A}\n\mc{B}\n\mc{B}^{\core}\n\mc{A}^{\core}=(\mc{A}\n\mc{B}\n\mc{B}^{\core}\n\mc{A}^{\core})^*=(\mc{A}^{\core})^*\n\mc{B}\n\mc{B}^{\core}\n\mc{A}^*,
\end{equation*}
\begin{eqnarray*}
(\mc{A}^{\core})^*\n\mc{B}\n(\mc{B}^{\core}\n\mc{A}^*)^2&=&(\mc{A}^{\core})^*\n\mc{B}\n\mc{B}^{\core}\n\mc{A}^*\n\mc{B}^{\core}\n\mc{A}^*\\
&=&(\mc{A}\n\mc{B}\n\mc{B}^{\core}\n\mc{A}^{\core})^*\n\mc{B}^{\core}\n\mc{A}^*=\mc{A}\n\mc{B}\n\mc{B}^{\core}\n\mc{A}^{\core}\n\mc{B}^{\core}\n\mc{A}^*\\
&=&\mc{A}\n\mc{B}\n\mc{B}^{\core}\n\mc{A}^{\core}\n\mc{B}^{\core}\n\mc{A}^{\core}\n\mc{A}\n\mc{A}^*=\mc{B}^{\core}\n\mc{A}^{\core}\n\mc{A}\n\mc{A}^*\\
&=&\mc{B}^{\core}\n\mc{A}^*,\mbox{ and }
\end{eqnarray*}
\begin{eqnarray*}
\mc{B}^{\core}\n\mc{A}^*\n((\mc{A}^{\core})^*\n\mc{B})^2&=&\mc{B}^{\core}\n\mc{A}^{\core}\n\mc{A}\n\mc{B}\n(\mc{A}^{\core})^*\n\mc{B}\\
&=&\mc{B}^{\core}\n\mc{A}^{\core}\n\mc{A}\n\mc{B}\n(\mc{A}\n\mc{A}^{\core}\n\mc{A}^{\core})^*\n\mc{B}\\
&=& \mc{B}^{\core}\n\mc{A}^{\core}\n\mc{A}\n\mc{B}\n\mc{A}\n\mc{A}^{\core}\n(\mc{A}^{\core})^*\n\mc{B}\\
&=&\mc{B}^{\core}\n\mc{A}^{\core}\n\mc{A}\n\mc{B}\n\mc{A}\n\mc{B}\n\mc{U}\n\mc{A}^{\core}\n(\mc{A}^{\core})^*\n\mc{B}\\
&=&\mc{A}\n\mc{B}\n\mc{U}\n\mc{A}^{\core}\n(\mc{A}^{\core})^*\n\mc{B}=\mc{A}\n\mc{A}^{\core}\n(\mc{A}^{\core})^*\n\mc{B}\\
&=&(\mc{A}^{\core})^*\n\mc{B}.
\end{eqnarray*}
Therefore, $\mc{B}^{\core}\n\mc{A}^*$ is the core inverse of $(\mc{A}^{\core})^*\n\mc{B}$.
\end{proof}
\begin{remark}
The condition $\rg(\mc{A})\subseteq\rg(\mc{A}\n\mc{B})$ can be replaced by $\rg(\mc{B}^*)\subseteq\rg(\mc{B}^*\n\mc{A}^*)$.
\end{remark}
\begin{corollary}\label{cor3.12}
Let $\mc{A},~\mc{B}\in\C^{\textbf{N}(n)\times \textbf{N}(n)}$ be core tensors. Assume that $\rg(\mc{A})\subseteq\rg(\mc{A}\n\mc{B})$, $\rg(\mc{B}^*)\subseteq\rg(\mc{B}^*\n\mc{A}^*)$ and $\mc{A}$ is EP.  If $(\mc{A}\n\mc{B})^{\core} = \mc{B}^{\core}\n\mc{A}^{\core}$, then  
\begin{center}
    $(\mc{A}\n\mc{B}\n\mc{B}^{\core})^{\core}=\mc{B}\n\mc{B}^{\core}\n\mc{A}^{\core}$.
\end{center}

\end{corollary}
\begin{proof}
By Theorem \ref{thm3.4}, it is sufficient to show only $\mc{B}\n\mc{B}^{\core}\n\mc{A}^{\core}\in\mc{C}\{7\}$, where $\mc{C}=\mc{A}\n\mc{B}\n\mc{B}^{\core}$. From the range conditions, we have $\mc{B}=\mc{U}\n\mc{A}\n\mc{B}$ and $\mc{A}=\mc{A}\n\mc{B}\n\mc{V}$ for some $\mc{U}\in\C^{\textbf{N}(n)\times \textbf{N}(n)}$ and $\mc{V}\in\C^{\textbf{N}(n)\times \textbf{N}(n)}$. Applying these results along with Eq. \ref{eqar} and $\mc{A}\n\mc{A}^{\core}=\mc{A}^{\core}\n\mc{A}$, we obtain
\begin{eqnarray*}
\mc{C}\n(\mc{B}\n\mc{B}^{\core}\n\mc{A}^{\core})^2&=&\mc{A}\n\mc{B}\n\mc{B}^{\core}\n\mc{A}^{\core}\n\mc{B}\n\mc{B}^{\core}\n\mc{A}^{\core}\\
&=&\mc{B}\n\mc{B}^{\core}\n\mc{A}\n\mc{B}\n\mc{B}^{\core}\n\mc{A}^{\core}\n\mc{A}\n\mc{A}^{\core}\n\mc{B}\n\mc{B}^{\core}\n\mc{A}^{\core}\\
&=&\mc{B}\n\mc{B}^{\core}\n(\mc{A}\n\mc{B}\n\mc{B}^{\core}\n\mc{A}^{\core}\n\mc{A}\n\mc{B})\n\mc{V}\n\mc{A}^{\core}\n\mc{B}\n\mc{B}^{\core}\n\mc{A}^{\core}\\
&=&\mc{U}\n\mc{A}\n\mc{B}\n\mc{B}^{\core}\n(\mc{A}\n\mc{B}\n\mc{V})\n\mc{A}^{\core}\n\mc{B}\n\mc{B}^{\core}\n\mc{A}^{\core}\\
&=&\mc{U}\n\mc{A}\n\mc{B}\n\mc{B}^{\core}\n\mc{A}\n\mc{A}^{\core}\n\mc{B}\n\mc{B}^{\core}\n\mc{A}^{\core}\\
&=&\mc{U}\n\mc{A}\n\mc{B}\n\mc{B}^{\core}\mc{A}^{\core}\n\mc{A}\n\mc{B})\n\mc{B}^{\core}\n\mc{A}^{\core}\\
&=&\mc{U}\n\mc{A}\n\mc{B}\n\mc{B}^{\core}\n\mc{A}^{\core}=\mc{B}\n\mc{B}^{\core}\n\mc{A}^{\core}.
\end{eqnarray*}
Hence $\mc{B}\n\mc{B}^{\core}\n\mc{A}^{\core}\in\mc{C}\{7\}$.
\end{proof}

\section{Mixed-type of reverse-order law}
In this section, we present a few results related to various equivalents of the mixed-type reverse-order law for the core inverse of tensors. 
\begin{theorem}\label{thm3.14}
Let $\mc{A},~\mc{B}\in\C^{\textbf{N}(n)\times \textbf{N}(n)}$ and $ind(\mc{B})=ind(\mc{A}\n\mc{B})=ind(\mc{A}\n\mc{B}\n\mc{B}^{\core})=1$. Then the following statements are equivalent:
\begin{enumerate}
    \item[(a)] $(\mc{A}\n\mc{B})^{\core} =(\mc{A}\n\mc{B})^{\#}= \mc{B}^{\core}\n(\mc{A}\n\mc{B}\n\mc{B}^{\core})^{\core}$;
    \item[(b)] $\mc{B}^{\core}\n(\mc{A}\n\mc{B}\n\mc{B}^{\core})^{\core}\in(\mc{A}\n\mc{B})\{5\}$;
    \item[(c)]  $\mc{B}^{\core}\n\mc{B}\n\mc{A}\n\mc{B}=\mc{A}\n\mc{B}=\mc{B}^{\core}\n\mc{B}\n\mc{A}\n\mc{B}$ and $\mc{B}\n\mc{A}\n(\mc{A}\n\mc{B}\n\mc{B}^{\core})^{\core}=(\mc{A}\n\mc{B}\n\mc{B}^{\core})^{\core}\n\mc{A}\n\mc{B}$
    \item[(d)] $(\mc{A}\n\mc{B}\n\mc{B}^{\core})^{\core}=\mc{B}\n(\mc{A}\n\mc{B})^{\#}$
\end{enumerate}
\end{theorem}
\begin{proof}
$(a)\Rightarrow (b)$ Since $\mc{B}^{\core}\n(\mc{A}\n\mc{B}\n\mc{B}^{\core})^{\core}=(\mc{A}\n\mc{B})^{\#}$, so  $(b)$ is trivial.\\
$(b)\Rightarrow (c)$ Let $\mc{B}^{\core}\n(\mc{A}\n\mc{B}\n\mc{B}^{\core})^{\core}\in(\mc{A}\n\mc{B})\{5\}$. Which implies $\mc{B}^{\core}\n(\mc{A}\n\mc{B}\n\mc{B}^{\core})^{\core}$ and $\mc{A}\n\mc{B}$ commutes each other. Now
\begin{eqnarray*}
\mc{A}\n\mc{B}&=&\mc{A}\n\mc{B}\n\mc{B}^{\core}\n\mc{B}=(\mc{A}\n\mc{B}\n\mc{B}^{\core}\n(\mc{A}\n\mc{B}\n\mc{B}^{\core})^{\core}\n\mc{A}\n\mc{B}\n\mc{B}^{\core})\n\mc{B}\\
&=&\mc{A}\n\mc{B}\n\mc{B}^{\core}\n(\mc{A}\n\mc{B}\n\mc{B}^{\core})^{\core}\n\mc{A}\n\mc{B}=\mc{B}^{\core}\n(\mc{A}\n\mc{B}\n\mc{B}^{\core})^{\core}\n(\mc{A}\n\mc{B})^2.
\end{eqnarray*}
Using the above expression of $\mc{A}\n\mc{B}$ repetitively, we obtain 
\begin{equation*}
 \mc{A}\n\mc{B}=\mc{B}^{\core}\n\mc{B}\n\mc{B}^{\core}\n(\mc{A}\n\mc{B}\n\mc{B}^{\core})^{\core}\n(\mc{A}\n\mc{B})^2 =\mc{B}^{\core}\n\mc{B}\n\mc{A}\n\mc{B}, \mbox{ and}
\end{equation*}
\begin{equation}\label{eq15}
 \mc{A}\n\mc{B}=\mc{B}\n\mc{B}^{\core}\n\mc{B}^{\core}\n(\mc{A}\n\mc{B}\n\mc{B}^{\core})^{\core}\n(\mc{A}\n\mc{B})^2 =\mc{B}\n\mc{B}^{\core}\n\mc{A}\n\mc{B}.
\end{equation}
From the definition of core inverse, we have $\rg((\mc{A}\n\mc{B}\n\mc{B}^{\core})^{\core})\subseteq \rg(\mc{A}\n\mc{B}\n\mc{B}^{\core})$. This yields $(\mc{A}\n\mc{B}\n\mc{B}^{\core})^{\core}=\mc{A}\n\mc{B}\n\mc{B}^{\core}\n\mc{U}$ for some $\mc{U}\in\C^{\textbf{N}(n)\times \textbf{N}(n)}$. Also from Eq. \eqref{eq15}, we have $\mc{A}\n\mc{B}\n\mc{B}^{\core}=\mc{B}\n\mc{B}^{\core}\n\mc{A}\n\mc{B}\n\mc{B}^{\core}$. Applying these results, we obtain
\begin{eqnarray*}
\mc{B}\n\mc{A}\n(\mc{A}\n\mc{B}\n\mc{B}^{\core})^{\core}&=&\mc{B}\n\mc{A}\n\mc{A}\n\mc{B}\n\mc{B}^{\core}\n\mc{U}=\mc{B}\n\mc{A}\n\mc{B}\n\mc{B}^{\core}\n\mc{A}\n\mc{B}\n\mc{B}^{\core}\n\mc{U}\\
&=&\mc{B}\n\mc{A}\n\mc{B}\n\mc{B}^{\core}\n(\mc{A}\n\mc{B}\n\mc{B}^{\core})^{\core}\\
&=&\mc{B}\n\mc{B}^{\core}\n(\mc{A}\n\mc{B}\n\mc{B}^{\core})^{\core}\n\mc{A}\n\mc{B}\\
&=&\mc{B}\n\mc{B}^{\core}\n(\mc{A}\n\mc{B}\n\mc{B}^{\core}\n\mc{U})\n\mc{A}\n\mc{B}\\
&=& (\mc{A}\n\mc{B}\n\mc{B}^{\core}\n\mc{U})\n\mc{A}\n\mc{B}=(\mc{A}\n\mc{B}\n\mc{B}^{\core})^{\core}\n\mc{A}\n\mc{B}.
\end{eqnarray*}
$(c)\Rightarrow (d)$ Let $\mc{A}\n\mc{B}=\mc{B}\n\mc{B}^{\core}\n\mc{A}\n\mc{B}$. From the range condition $\rg((\mc{A}\n\mc{B}\n\mc{B}^{\core})^{\core})\subseteq \rg(\mc{A}\n\mc{B}\n\mc{B}^{\core})$, we have $(\mc{A}\n\mc{B}\n\mc{B}^{\core})^{\core}=\mc{B}\n\mc{B}^{\core}\n(\mc{A}\n\mc{B}\n\mc{B}^{\core})^{\core}$. So, it is enough to show $\mc{B}^{\core}\n(\mc{A}\n\mc{B}\n\mc{B}^{\core})^{\core}$ is the group inverse of $\mc{A}\n\mc{B}$. Since
\begin{equation*}
 \mc{B}^{\core}\n(\mc{A}\n\mc{B}\n\mc{B}^{\core})^{\core}\n(\mc{A}\n\mc{B})\n\mc{B}^{\core}\n(\mc{A}\n\mc{B}\n\mc{B}^{\core})^{\core}=\mc{B}^{\core}\n(\mc{A}\n\mc{B}\n\mc{B}^{\core})^{\core},  
\end{equation*}
\begin{equation*}
    \mc{A}\n\mc{B}\n(\mc{B}^{\core}\n(\mc{A}\n\mc{B}\n\mc{B}^{\core})^{\core})\n\mc{A}\n\mc{B}=\mc{A}\n\mc{B}\n\mc{B}^{\core}\n\mc{B}=\mc{A}\n\mc{B},\mbox{ and }
\end{equation*}
\begin{eqnarray*}
(\mc{A}\n\mc{B})\n(\mc{B}^{\core}\n(\mc{A}\n\mc{B}\n\mc{B}^{\core})^{\core})&=& \mc{B}^{\core}\n\mc{B}\n\mc{A}\n\mc{B}\n(\mc{B}^{\core}\n(\mc{A}\n\mc{B}\n\mc{B}^{\core})^{\core})\\
&=&\mc{B}^{\core}\n\mc{B}\n\mc{A}\n(\mc{A}\n\mc{B}\n\mc{B}^{\core})^{\core}\\
&=&\mc{B}^{\core}\n(\mc{A}\n\mc{B}\n\mc{B}^{\core})^{\core}\n\mc{A}\n\mc{B}.
\end{eqnarray*}
Therefore, $\mc{B}\n(\mc{A}\n\mc{B})^{\#}=\mc{B}\n\mc{B}^{\core}\n(\mc{A}\n\mc{B}\n\mc{B}^{\core})^{\core}=(\mc{A}\n\mc{B}\n\mc{B}^{\core})^{\core}$.\\
$(d)\Rightarrow (a)$ Let $(\mc{A}\n\mc{B}\n\mc{B}^{\core})^{\core}=\mc{B}\n(\mc{A}\n\mc{B})^{\#}$. Using this, we obtain 
\begin{eqnarray*}
\mc{A}\n\mc{B}&=&\mc{A}\n\mc{B}\n\mc{B}^{\core}\n\mc{B}=(\mc{A}\n\mc{B}\n\mc{B}^{\core})^{\core}\n(\mc{A}\n\mc{B}\n\mc{B}^{\core})^2\n\mc{B}\\
&=&\mc{B}\n(\mc{A}\n\mc{B})^{\#}\n(\mc{A}\n\mc{B}\n\mc{B}^{\core})^2\n\mc{B}=\mc{B}^{\core}\n\mc{B}^2\n(\mc{A}\n\mc{B})^{\#}\n(\mc{A}\n\mc{B}\n\mc{B}^{\core})^2\n\mc{B}\\
&=&\mc{B}^{\core}\n\mc{B}\n((\mc{A}\n\mc{B}\n\mc{B}^{\core})^{\core})\n(\mc{A}\n\mc{B}\n\mc{B}^{\core})^2\n\mc{B}\\
&=&\mc{B}^{\core}\n\mc{B}\n\mc{A}\n\mc{B}\n\mc{B}^{\core}\n\mc{B}=\mc{B}^{\core}\n\mc{B}\n\mc{A}\n\mc{B}.
\end{eqnarray*}
Thus $(\mc{A}\n\mc{B})^{\#}=\mc{A}\n\mc{B}\n((\mc{A}\n\mc{B})^{\#})^2=\mc{B}^{\core}\n\mc{B}\n\mc{A}\n\mc{B}\n((\mc{A}\n\mc{B})^{\#})^2=\mc{B}^{\core}\n\mc{B}\n(\mc{A}\n\mc{B})^{\#}=\mc{B}^{\core}\n(\mc{A}\n\mc{B}\n\mc{B}^{\core})^{\core}$. Next we will claim that $\mc{B}^{\core}\n(\mc{A}\n\mc{B}\n\mc{B}^{\core})^{\core}$ is the core inverse of $\mc{A}\n\mc{B}$.  Since $(\mc{A}\n\mc{B}\n\mc{B}^{\core}\n(\mc{A}\n\mc{B}\n\mc{B}^{\core})^{\core})^*=\mc{A}\n\mc{B}\n\mc{B}^{\core}\n(\mc{A}\n\mc{B}\n\mc{B}^{\core})^{\core}$, $\mc{A}\n\mc{B}=(\mc{A}\n\mc{B})^{\#}\n(\mc{A}\n\mc{B})^2=\mc{B}^{\core}\n(\mc{A}\n\mc{B}\n\mc{B}^{\core})^{\core}\n(\mc{A}\n\mc{B})^2$, and $\mc{A}\n\mc{B}\n(\mc{B}^{\core}\n(\mc{A}\n\mc{B}\n\mc{B}^{\core})^{\core})^2=\mc{A}\n\mc{B}\n((\mc{A}\n\mc{B})^{\#})^2=(\mc{A}\n\mc{B})^{\#}=\mc{B}^{\core}\n(\mc{A}\n\mc{B}\n\mc{B}^{\core})^{\core}$. Hence completes the proof.
\end{proof}

\begin{theorem}\label{thm3.15}
Let $\mc{A},~\mc{B}\in\C^{\textbf{N}(n)\times \textbf{N}(n)}$ and $ind(\mc{A})=ind(\mc{B})$. Then the following statements are equivalent:
\begin{enumerate}
    \item[(a)] $\mc{A}^{\core}=\mc{B}\n(\mc{A}\n\mc{B})^{\#}$;
    \item[(b)] $\mc{A}^{\core}\n\mc{A}\n\mc{B}=\mc{B}\n\mc{A}\n\mc{A}^{\core}$ and $\rg(\mc{A})\subseteq\rg(\mc{A}\n\mc{B})$;
    \item[(c)] $(\mc{A}\n\mc{B})^{\#}=(\mc{A}\n\mc{B})^{\core}=\mc{B}^{\core}\n\mc{A}^{\core}$ and $\rg(\mc{A})\subseteq\rg(\mc{B})$.
   \end{enumerate}
\end{theorem}

\begin{proof}
$(a)\Rightarrow (b)$ Let $\mc{A}^{\core}=\mc{B}\n(\mc{A}\n\mc{B})^{\#}$. Since $\mc{A}=\mc{A}\n\mc{A}^{\core}\n\mc{A}=\mc{A}\n\mc{B}\n(\mc{A}\n\mc{B})^{\#}\n\mc{A}$, so it implies $\rg(\mc{A})\subseteq\rg(\mc{A}\n\mc{B})$. Further, we get the second part as follows,
\begin{center}
    $\mc{A}^{\core}\n\mc{A}\n\mc{B}=\mc{B}\n(\mc{A}\n\mc{B})^{\#}\n\mc{A}\n\mc{B}=\mc{B}\n\n\mc{A}\n\mc{B}\n(\mc{A}\n\mc{B})^{\#}=\mc{B}\n\n\mc{A}\n\mc{A}^{\core}$.
\end{center}
$(b)\Rightarrow (c)$ Let $\rg(\mc{A})\subseteq\rg(\mc{A}\n\mc{B})$.  This yields $\mc{A}=\mc{A}\n\mc{B}\n\mc{U}$ for some $\mc{U}\in\C^{\textbf{N}(n)\times \textbf{N}(n)}$. Further, $\mc{A}=\mc{A}^{\core}\n\mc{A}^2=\mc{A}^{\core}\n\mc{A}\n\mc{B}\n\mc{U}\n\mc{A}=\mc{B}\n\mc{A}\n\mc{A}^{\core}\n\mc{U}\n\mc{A}$. Thus $\rg(\mc{A})\subseteq\rg(\mc{B})$ and hence $\mc{A}=\mc{B}\n\mc{V}$ for some $\mc{V}\in\C^{\textbf{N}(n)\times \textbf{N}(n)}$. By using $\mc{A}=\mc{B}\n\mc{V}$, we obtain 
\begin{equation*}
\mc{B}^{\core}\n\mc{A}^{\core}\n\mc{A}\n\mc{B}=\mc{B}^{\core}\n\mc{B}\n\mc{A}\n\mc{A}^{\core}=\mc{B}^{\core}\n\mc{B}^2\n\mc{V}\n\mc{A}^{\core}=\mc{B}\n\mc{V}\n\mc{A}^{\core}=\mc{A}\n\mc{A}^{\core},
\end{equation*}
and 
\begin{equation*}
\mc{A}\n\mc{B}\n\mc{B}^{\core}\n\mc{A}^{\core}=\mc{A}\n\mc{B}\n\mc{B}^{\core}\n\mc{A}\n(\mc{A}^{\core})^2=\mc{A}\n\mc{B}\n\mc{V}\n(\mc{A}^{\core})^2=\mc{A}^2\n(\mc{A}^{\core})^2=\mc{A}\n\mc{A}^{\core}.
\end{equation*}
Therefore, $\mc{B}^{\core}\n\mc{A}^{\core}\in(\mc{A}\n\mc{B})\{3,5\}$. Further $\mc{A}\n\mc{B}\n\mc{B}^{\core}\n\mc{A}^{\core}\n \mc{A}\n\mc{B}=\mc{A}\n\mc{A}^{\core}\n \mc{A}\n\mc{B}=\mc{A}\n\mc{B}$ and $\mc{B}^{\core}\n\mc{A}^{\core}\mc{A}\n\mc{B}\n\mc{B}^{\core}\n\mc{A}^{\core}=\mc{B}^{\core}\n\mc{A}^{\core}\n\mc{A}\n\mc{A}^{\core}=\mc{B}^{\core}\n\mc{A}^{\core}$. Hence $\mc{B}^{\core}\n\mc{A}^{\core}=(\mc{A}\n\mc{B})^{\#}$. Since  $\mc{B}^{\core}\n\mc{A}^{\core}\n(\mc{A}\n\mc{B})^2=(\mc{A}\n\mc{B})^{\#}\n(\mc{A}\n\mc{B})^2=\mc{A}\n\mc{B}$ and 
\begin{equation*}
   \mc{A}\n\mc{B}\n(\mc{B}^{\core}\n\mc{A}^{\core})^2=\mc{A}\n\mc{B}\n((\mc{A}\n\mc{B})^{\#})^2=(\mc{A}\n\mc{B})^{\#}=\mc{B}^{\core}\n\mc{A}^{\core}, 
\end{equation*}
we have $(\mc{A}\n\mc{B})^{\core}=\mc{B}^{\core}\n\mc{A}^{\core}$.\\
$(c)\Rightarrow (a)$  Let $\mc{X}=\mc{B}\n(\mc{A}\n\mc{B})^{\#}$ and $(\mc{A}\n\mc{B})^{\core}=(\mc{A}\n\mc{B})^{\#}$. Then 
\begin{center}
  $\mc{X}\n\mc{A}\mc{X}=\mc{B}\n(\mc{A}\n\mc{B})^{\#}\n\mc{A}\n\mc{B}\n(\mc{A}\n\mc{B})^{\#}=\mc{B}\n(\mc{A}\n\mc{B})^{\#}=\mc{X}$,   and
\end{center}
 $(\mc{A}\n\mc{X})^*=(\mc{A}\n\mc{B}\n(\mc{A}\n\mc{B})^{\#})^*=(\mc{A}\n\mc{B}\n(\mc{A}\n\mc{B})^{\core})^*=\mc{A}\n\mc{B}\n(\mc{A}\n\mc{B})^{\core}=\mc{A}\n\mc{X}$.  From the  condition $\rg(\mc{A})\subseteq\rg(\mc{B})$, we have $\mc{X}\n\mc{A}^2=\mc{B}\n(\mc{A}\n\mc{B})^{\#}\n\mc{A}^2=\mc{B}\n\mc{B}^{\core}\n\mc{A}^{\core}\n\mc{A}^2=\mc{B}\n\mc{B}^{\core}\n\mc{A}=\mc{B}\n\mc{B}^{\core}\n\mc{B}\n\mc{V}=\mc{B}\n\mc{V}=\mc{A}$. By Lemma \ref{lm2.11}, $\mc{X}$ is the core inverse of $\mc{A}$.
\end{proof}

\begin{theorem}\label{thm3.16}
Let $\mc{A},~\mc{B}\in\C^{\textbf{N}(n)\times \textbf{N}(n)}$ and $ind(\mc{A})=1$. Then $\mc{A}^{\core}=\mc{B}\n(\mc{A}\n\mc{B})^{\core}$ if and only if $\rg(\mc{A})\subseteq\rg(\mc{B}\n\mc{A}\n\mc{B})$.
\end{theorem}
\begin{proof}
Let $\mc{A}^{\core}=\mc{B}\n(\mc{A}\n\mc{B})^{\core}$. Since $\mc{A}=\mc{A}^{\core}\n\mc{A}^{\core}\n\mc{A}=\mc{B}\n(\mc{A}\n\mc{B})^{\core}\n\mc{A}^{\core}\n\mc{A}=\mc{B}\n\mc{A}\n\mc{B}\n((\mc{A}\n\mc{B})^{\core})^2\n\mc{A}^{\core}\n\mc{A}$, we have  $\rg(\mc{A})\subseteq\rg(\mc{B}\n\mc{A}\n\mc{B})$.

Conversely let $\rg(\mc{A})\subseteq\rg(\mc{B}\n\mc{A}\n\mc{B})$. This implies $\mc{A}=\mc{B}\n\mc{A}\n\mc{B}\n\mc{U}$ for some $\mc{U}\in\C^{\textbf{N}(n)\times \textbf{N}(n)}$. Also From the range conditions  $\rg(\mc{A})\subseteq\rg(\mc{B}\n\mc{A}\n\mc{B})$ and  $\rg(\mc{A})=\rg(\mc{A}^2)$, it follows that $\rg(\mc{A}\n\mc{B})\subseteq\rg((\mc{A}\n\mc{B})^2)$. Hence $ind(\mc{A}\n\mc{B})=1$. Let $\mc{X}=\mc{B}\n(\mc{A}\n\mc{B})^{\core}$. Now 
\begin{center}
    $\mc{X}\n\mc{A}\n\mc{X}=\mc{B}\n(\mc{A}\n\mc{B})^{\core}\n\mc{A}\n\mc{B}\n(\mc{A}\n\mc{B})^{\core}=\mc{B}\n(\mc{A}\n\mc{B})^{\core}=\mc{X}$, 
\end{center}
\begin{center}
  $(\mc{A}\n\mc{X})^*=(\mc{A}\n\mc{B}\n(\mc{A}\n\mc{B})^{\core})^*=\mc{A}\n\mc{B}\n(\mc{A}\n\mc{B})^{\core}=\mc{A}\n\mc{X}$, and   
\end{center}
\begin{center}
    $\mc{X}\n\mc{A}^2=\mc{B}\n(\mc{A}\n\mc{B})^{\core}\n\mc{A}\n\mc{B}\n\mc{A}\n\mc{B}\n\mc{U}=n\mc{B}\n\mc{A}\n\mc{B}\n\mc{U}=\mc{A}$.
\end{center}
Thus by Lemma \ref{lm2.11}, $\mc{A}^{\core}=\mc{X}=\mc{B}\n(\mc{A}\n\mc{B})^{\core}$.
\end{proof}

\begin{theorem}\label{thm3.17}
Let $\mc{A},~\mc{B}\in\C^{\textbf{N}(n)\times \textbf{N}(n)}$ be core tensors. Then the following are equivalent:
\begin{enumerate}
    \item[(a)] $(\mc{A}\n\mc{B})^{\#}=\mc{B}^{\core}\n\mc{A}^{\core}$ and $(\mc{B}\n\mc{A})^{\#}=\mc{A}^{\core}\n\mc{B}^{\core}$;
    \item[(b)] $\mc{A}^{\core}\n\mc{A}\n\mc{B}=\mc{B}\n\mc{A}\n\mc{A}^{\core}$, $\mc{B}^{\core}\n\mc{B}\n\mc{A}=\mc{A}\n\mc{B}\n\mc{B}^{\core}$, and $\mc{A}\n\mc{B}^{\core}\n\mc{A}^{\core}=\mc{B}^{\core}\n\mc{A}^{\core}\n\mc{A}$.
   \end{enumerate}
\end{theorem}
\begin{proof}
$(a)\Rightarrow (b)$
From $(\mc{A}\n\mc{B})^{\#}=\mc{B}^{\core}\n\mc{A}^{\core}$, we obtain 
\begin{eqnarray*}
    \mc{A}\n\mc{B}&=&(\mc{A}\n\mc{B})^{\#}\n(\mc{A}\n\mc{B})^{2}=\mc{B}^{\core}\n\mc{A}^{\core}\n(\mc{A}\n\mc{B})^{2}=\mc{B}^{\core}\n\mc{B}\n(\mc{A}\n\mc{B})^{\#}\n(\mc{A}\n\mc{B})^2\\
    &=&\mc{B}^{\core}\n\mc{B}\n\mc{A}\n\mc{B}, \mbox{  and}
\end{eqnarray*}
\begin{center}
 $  \mc{A}\n\mc{B}=(\mc{A}\n\mc{B})^{2}\n(\mc{A}\n\mc{B})^{\#}=(\mc{A}\n\mc{B})^{2}\n\mc{B}^{\core}\n\mc{A}^{\core}\mc{A}\n\mc{A}^{\core}=\mc{A}\n\mc{B}\n\mc{A}\n\mc{A}^{\core}$.
 \end{center} 
 Similarly, from $(\mc{B}\n\mc{A})^{\#}=\mc{A}^{\core}\n\mc{B}^{\core}$, we get $\mc{B}\n\mc{A}=\mc{A}^{\core}\n\mc{A}\n\mc{B}\n\mc{A}=\mc{B}\n\mc{A}\n\mc{B}\n\mc{B}^{\core}$. Using the above properties of $\mc{A}\n\mc{B}$ and $\mc{B}\n\mc{A}$, we have  $\mc{A}^{\core}\n\mc{A}\n\mc{B}=\mc{A}^{\core}\n\mc{A}\n\mc{B}\n\mc{A}\mc{A}^{\core}=\mc{B}\n\mc{A}\mc{A}^{\core}$ and $\mc{B}^{\core}\n\mc{B}\n\mc{A}=\mc{B}^{\core}\n\mc{B}\n\mc{A}\n\mc{B}\n\mc{B}^{\core}=\mc{A}\n\mc{B}\n\mc{B}^{\core}$. Further by Lemma \ref{lm2.4}, 
 \begin{center}
 $\mc{A}\n\mc{A}^{\core}\n\mc{B}^{\core}=\mc{A}\n(\mc{B}\n\mc{A})^{\#}=\mc{A}\n\mc{B}((\mc{A}\n\mc{B})^{\#})^2\n\mc{A}=(\mc{A}\n\mc{B})^{\#}\n\mc{A}=\mc{B}^{\core}\n\mc{A}^{\core}\n\mc{A}$.
\end{center}
 $(b)\Rightarrow (a)$ By using the assumption of $(b)$, we obtain 
 \begin{eqnarray*}
 \mc{A}\n\mc{B}\n\mc{B}^{\core}\n\mc{A}^{\core}\n\mc{A}\n\mc{B}&=&\mc{A}\n\mc{B}\n\mc{B}^{\core}\n\mc{B}\n\mc{A}\n\mc{A}^{\core}=\mc{A}\n\mc{B}\n\mc{A}\n\mc{A}^{\core}=\\
 &=&\mc{A}\n\mc{A}^{\core}\n\mc{A}\n\mc{B}=\mc{A}\n\mc{B},\\
 \mc{B}^{\core}\n\mc{A}^{\core}\n\mc{A}\n\mc{B}\n\mc{B}^{\core}\n\mc{A}^{\core}&=&\mc{A}\n\mc{A}^{\core}\n\mc{B}^{\core}\n\mc{B}\n\n\mc{B}^{\core}\n\mc{A}^{\core}=\mc{A}\n\mc{A}^{\core}\n\n\mc{B}^{\core}\n\mc{A}^{\core}\\
 &=&\mc{B}^{\core}\n\mc{A}^{\core}\n\mc{A}\n\mc{A}^{\core}=\mc{B}^{\core}\n\mc{A}^{\core}, \mbox{ and }
  \end{eqnarray*}
\begin{center}
    $\mc{B}^{\core}\n\mc{A}^{\core}\n\mc{A}\n\mc{B}=\mc{B}^{\core}\n\mc{B}\n\mc{A}\n\mc{A}^{\core}=\mc{A}\n\mc{B}\n\mc{B}^{\core}\n\mc{A}^{\core}$.
\end{center}
Thus $\mc{B}^{\core}\n\mc{A}^{\core}$ is the group inverse of $(\mc{A}\n\mc{B})$. Similarly, we can show $(\mc{B}\n\mc{A})^{\#}=\mc{A}^{\core}\n\mc{B}^{\core}$.
\end{proof}

The core inverse of the Kronecker product of two tensors can be computed using the next result.

\begin{theorem}\label{revkron}
$\mc{A},~\mc{B}\in\C^{\textbf{N}(n)\times \textbf{N}(n)}$ be core tensors. Then $(\mc{A}\otimes\mc{B})^{\core} =
\mc{A}^{\core}\otimes\mc{B}^{\core}$.
\end{theorem}

\begin{proof}
Let $\mc{C} = \mc{A}\otimes\mc{B}$ and $\mc{X} =
\mc{A}^{\core}\otimes\mc{B}^{\core}$. By applying Poeposition \ref{Krop}, we obtain 
\begin{equation*}
 (\mc{C}\n\mc{X})^*=(\mc{A}\n\mc{A}^{\core})^*{\otimes} (\mc{B}\n\mc{B}^{\core})^*= \mc{A}\n\mc{A}^{\core}\otimes \mc{B}\n\mc{B}^{\core}=\mc{C}\n\mc{X},
\end{equation*}
\begin{center}
   $ \mc{C}\n\mc{X}^2=(\mc{A}\n\mc{A}^{\core}\otimes \mc{B}\n\mc{B}^{\core})\n(\mc{A}^{\core}\otimes\mc{B}^{\core})=\mc{A}\n(\mc{A}^{\core})^2\otimes\mc{A}\n(\mc{A}^{\core})^2=\mc{A}^{\core}\otimes\mc{B}^{\core}=\mc{X}$, 
\end{center}
and
\begin{center}
    $ \mc{X}\n\mc{C}^2=(\mc{A}^{\core}\n\mc{A}\otimes \mc{B}^{\core}\n\mc{B})\n(\mc{A}\otimes\mc{B})=(\mc{A}^{\core}\n\mc{A}^2)\otimes(\mc{B}^{\core}\n\mc{B}^2)=\mc{A}\otimes\mc{B}=\mc{C}
$.
\end{center}
\end{proof}

\begin{theorem}\label{unit}
Let $\mc{A},~\mc{B}\in\C^{\textbf{N}(n)\times \textbf{N}(n)}$   be two core tensors and $ind(\mc{A}\n\mc{B})=1$.
\begin{enumerate}
    \item[(a)] If $\mc{B}$ is unitary and $\rg(\mc{B}^*\n\mc{A}^{\core})\subseteq\rg(\mc{A}^{\core})$, then $(\mc{A}\n\mc{B})^{\core}=\mc{B}^*\n\mc{A}^{\core}$.
    \item[(b)] If $\mc{A}$ is unitary and $\rg(\mc{A})\subseteq\rg(B)$, then $(\mc{A}\n\mc{B})^{\core}=\mc{B}^{\core}\n\mc{A}^{*}$.
\end{enumerate}
\begin{proof}
$(a)$ Let $\rg(\mc{B}^*\n\mc{A}^{\core})\subseteq\rg(\mc{A}^{\core})$. This implies $\mc{B}^*\n\mc{A}^{\core}=\mc{A}^{\core}\n\mc{U}$ for some $\mc{U}\in\C^{\textbf{N}(n)\times \textbf{N}(n)}$. Now
\begin{center}
   $\mc{A}\n\mc{B}\n\mc{B}^*\n\mc{A}^{\core}\n\mc{A}\n\mc{B}=\mc{A}\n\mc{A}^{\core}\n\mc{A}\n\mc{B}=\mc{A}\n\mc{B}$,  \\
   $(\mc{A}\n\mc{B}\n\mc{B}^*\n\mc{A}^{\core})^*=(\mc{A}\n\mc{A}^{\core})^*=\mc{A}\n\mc{A}^{\core}=\mc{A}\n\mc{B}\n\mc{B}^*\n\mc{A}$, and 
\end{center}
\begin{eqnarray*}
\mc{A}\n\mc{B}\n(\mc{B}^*\n\mc{A}^{\core})^2=\mc{A}\n\mc{A}^{\core}\n\mc{B}^*\n\mc{A}^{\core}=\mc{A}\n\mc{A}^{\core}\n\mc{A}^{\core}\n\mc{U}=\mc{A}^{\core}\n\mc{U}=\mc{B}^*\n\mc{A}^{\core}.
\end{eqnarray*}
Thus by Lemma \ref{lm2.11}, $(\mc{A}\n\mc{B})^{\core}=\mc{B}^*\n\mc{A}^{\core}$.\\
$(b)$ From the range condition $\rg(\mc{A})\subseteq\rg(\mc{B})$, we have $\mc{A}=\mc{B}\n\mc{U}$ for some $\mc{U}\in\C^{\textbf{N}(n)\times \textbf{N}(n)}$. Now $\mc{B}^{\core}\n\mc{A}^{*}\n\mc{A}\n\mc{B}\n\mc{B}^{\core}\n\mc{A}^{*}=\mc{B}^{\core}\n\n\mc{B}\n\mc{B}^{\core}\n\mc{A}^{*}=\mc{B}^{\core}\n\mc{A}^{*}$, $(\mc{A}\n\mc{B}\n\mc{B}^{\core}\n\mc{A}^{*})^*=\mc{A}\n\mc{B}\n\mc{B}^{\core}\n\mc{A}^{*}$, and 
$ \mc{B}^{\core}\n\mc{A}^{*}\n( \mc{A}\n\mc{B})^2=\mc{B}^{\core}\n\mc{B}\n\mc{A}\n\mc{B}=\mc{B}\n\mc{U}\n\mc{B}=\mc{A}\n\mc{B}
$. Again by Lemma \ref{lm2.11}, $(\mc{A}\n\mc{B})^{\core}=\mc{B}^{\core}\n\mc{A}^{*}$.
\end{proof}
\end{theorem}
\begin{remark}
If we replace invertibility in place of unitary, still the Theorem \ref{unit} holds.
\end{remark}

\section{Solution of Multilinear system}
The Sylvester tensor equation  via the Einstein  product  plays significant roles in multilinear system \cite{RD, sun}. One can  compute
the exact solution of multinear systems by using the Kronecker product. This can be written in the following way:
\begin{equation}\label{sy1}
\mc{C}*_N\mc{X} + \mc{X}*_M\mc{D} = \mc{B},
\end{equation}
where $ \mc{C} \in \mathbb{C}^{\textbf{I}(N)\times \textbf{I}(N)}, ~\mc{X} \in
\mathbb{C}^{\textbf{I}(N)\times \textbf{I}(M)}, ~\mc{D} \in \mathbb{C}^{\textbf{I}(M)\times \textbf{I}(M)}$, and $ \mc{B} \in
\mathbb{C}^{\textbf{I}(N)\times \textbf{I}(M)}$.  
Further, the representations of block tensor were discussed in \cite{RD, sun}, as follows. 
\begin{equation}\label{sy2}
\left[
\begin{array}{cc} \mc{C} & \mc{I}_1
\end{array}
\right] *_N
\begin{bmatrix}
 \mc{X} & \mc{O} \\ \mc{O} & \mc{X}
\end{bmatrix} *_N
\left[
\begin{array}{c} \mc{I}_2 \\ \mc{D}
\end{array}
\right] = \mc{B},
\end{equation}
where $\mc{I}_1 = \mathbb{C}^{\textbf{I}(N)\times \textbf{I}(N)}$ and $\mc{I}_2 =
\mathbb{C}^{\textbf{I}(M)\times \textbf{I}(M)}$ are unit tensors. Equivalently, we have

\begin{equation}\label{sy3}
\mc{C}*_N\mc{X}*_M\mc{D} = \mc{B},
\end{equation}
where $ \mc{C} = \mathbb{C}^{\textbf{I}(N)\times \textbf{J}(N)}, $ $ \mc{X} =
\mathbb{R}^{\textbf{J}(N)\times \textbf{K}(M)}, $ $ \mc{D} = \mathbb{C}^{\textbf{K}(M)\times \textbf{L}(M)} $ and $ \mc{B} =
\mathbb{C}^{\textbf{I}(N)\times \textbf{L}(M)}$. Further, the above equation can be rewritten as
\begin{equation}\label{sy4}
    (\mc{C}\otimes\mc{D}^*)\n\mc{X}=\mc{B},
\end{equation}
where $n=N+M$ and $\otimes$ is the Kronecker product 
Let $\mc{A} \in \mathbb{C}^{\textbf{N}(n) \times \textbf{N}(n) }$ and consider the following singular tensor equation 
\begin{equation}\label{eq1.10}
   \mc{A}\n\mc{X} = \mc{B}, ~~ \mc{X},~ \mc{B} \in \mathbb{C}^{\textbf{N}(n) }.
\end{equation}

If ind$(\mc{A})=1$ and the tensor $\mc{B}\in \rg(\mc{A}),$ then Eq. (\ref{eq1.10}) is called the consistent multilinear system and its solution, we call core inverse solution or simply solution. 
Such multilinear systems arise in numerous applications in computational science and engineering such as continuum physics and engineering, isotropic
and anisotropic elasticity \cite{lai}. Multilinear systems are also prevalent in solving PDEs numerically. 
 
 \begin{lemma}\label{lm5.1} \label{lm5.2}
Let $\mc{A}\in\mathbb{C}^{\textbf{N}(n)\times\textbf{N}(n)}$  be a core tensor.\\ 
(a)~ Then  $\mc{A}^{\core}\n\mc{B}$ is a solution of $(\ref{eq1.10})$ if and only if $\mc{B} \in \rg(\mc{A})$.\\
(b)~If $\mc{B} \in \rg(\mc{A})$, then the general solution of $(\ref{eq1.10})$ is of the form 
$$\mc{X} = \mc{A}^{\core}\n\mc{B}+(\mc{I}-\mc{A}^{\core}\n\mc{A})\n\mc{Z},$$ for any arbitrary tensor $\mc{Z}\in \mathbb{C}^{\textbf{N}(n)}+\nl(\mc{A})$.
\end{lemma}

\begin{proof}
$(a)$~Let $\mc{A}^{\core}\n\mc{B}$ is a solution of $\mc{A}\n\mc{X}=\mc{B}$. This implies $\mc{B}=\mc{A}\n\mc{A}^{\core}\n\mc{B}$. Thus $\mc{B}\in\rg(\mc{A})$. Conversely if $\mc{B}\in\rg(\mc{A})$. This yields $\mc{B}=\mc{A}\n\mc{U}$ for some $\mc{U}\in\mathbb{C}^{\textbf{N}(n)\times\textbf{N}(n)}$. Now $\mc{A}\n\mc{A}^{\core}\n\mc{B}=\mc{A}\n\mc{A}^{\core}\n\n\mc{A}\n\mc{U}=\mc{A}\n\mc{U}=\mc{B}$. Thus $\mc{A}^{\core}\n\mc{B}$ is a solution of $\mc{A}\n\mc{X}=\mc{B}$.

$(b)$~ First we will claim that $\mc{X}$ is a solution of the tensor equation \eqref{eq1.10}. Since $\mc{B}\in\rg(\mc{A})$, so by part $(a)$,  $\mc{A}^{\core}\n\mc{B}$ is a solution of \eqref{eq1.10}. That is $\mc{A}\n\mc{A}^{\core}\n\mc{B}=\mc{B}$. Let $\mc{X}=\mc{A}^{\core}\n\mc{B}+(\mc{I}-\mc{A}^{\core}\n\mc{A})\n(\mc{Z}_1+\mc{Z}_2)$, where $\mc{Z}_1\in\mathbb{C}^{\textbf{N}(n)}$ and $\mc{A}\n\mc{Z}_2=\mc{O}$. Now
\begin{center}
 $\mc{A}\n\mc{X}=\mc{A}\n\mc{A}^{\core}\n\mc{B}+(\mc{A}-\mc{A}\n\mc{A}^{\core}\n\mc{A})\n(\mc{Z}_1+\mc{Z}_2)=\mc{A}\n\mc{A}^{\core}\n\mc{B}+\mc{O}=\mc{B}$.   
\end{center}
Therefore, $\mc{X}$ is a solution of \eqref{eq1.10}. Let $\mc{Y}$ be any arbitrary solution of \eqref{eq1.10}. Clearly $\mc{Y}-\mc{A}^{\core}\n\mc{B}\in \nl(\mc{A})$ and $\rg(\mc{I}-\mc{A}^{\core}\n\mc{A})\subseteq \nl(\mc{A})$. Since $\nl(\mc{A})=\rg(\mc{I}-\mc{A}^{\core}\n\mc{A})+\left(\nl(\mc{A})\cap\rg(\mc{I}-\mc{A}^{\core}\n\mc{A})^\perp\right)$, we have $\mc{Y}-\mc{A}^{\core}\n\mc{B}=(\mc{I}-\mc{A}^{\core}\n\mc{A})\n\mc{Y}_1+\mc{Y}_2$, where $\mc{Y}_2\in \nl(\mc{A})\cap\rg(\mc{I}-\mc{A}^{\core}\n\mc{A})^\perp$. Since $\mc{Y}_2\in\nl(\mc{A})$, we have $\mc{A}\n\mc{Y}_2=\mc{O}$. Further we can write $\mc{Y}_2=\mc{Y}_2-\mc{A}^{\core}\n\mc{A}\n\mc{Y}_2=(\mc{I}-\mc{A}^{\core}\n\mc{A})\n\mc{Y}_2$. Thus $\mc{Y}-\mc{A}^{\core}\n\mc{B}=(\mc{I}-\mc{A}^{\core}\n\mc{A})\n(\mc{Y}_1+\mc{Y}_2)$, where $\mc{Y}_1\in\in\mathbb{C}^{\textbf{N}(n)}$ and $\mc{Y}_2\in\nl(\mc{A})$. Hence completes the proof.
\end{proof}

Now we discuss the general solution of \eqref{eq1.10}, as follows.

\begin{corollary}\label{kroapp}
Let  $\mc{A} \in \mathbb{C}^{\textbf{N}(n) \times \textbf{N}(n) }$ be a core tensor. Then
$\mc{A}\n \mc{X}=\mc{B}$ has a solution if and only if $\mc{A}\n \mc{A}^{\core}\n\mc{B} = \mc{B}$. If a solution exists, then every solution is of the form
\begin{equation}
\mc{X}=\mc{A}^{\core}\n\mc{B} + (\mc{I}- \mc{A}^{\core} \n \mc{A})\n\mc{Y},
\end{equation}
where $\mc{Y}$ is  arbitrary.
\end{corollary}

In view of Theorem \ref{revkron} and Lemma \ref{lm5.2}, we state the following result as a corollary.
\begin{corollary}\label{cor4.3}
Let  $\mc{C},~\mc{D}\in \mathbb{C}^{\textbf{N}(n) \times \textbf{N}(n) }$ and $ind(\mc{C}\otimes\mc{D}^*)=1$. If $\mc{B}\in \rg(\mc{C}\otimes\mc{D}^*)$, then the general solution of \eqref{sy3} is of the form $\mc{X} = \left(\mc{C}^{\core}\otimes (\mc{D}^*)^{\core}\right)\n\mc{B}+\left(\mc{I}-(\mc{C}^{\core}\n\mc{C})\otimes(\mc{D}^*)^{\core}\n\mc{D}^*)\right)\n\mc{Z}$, where $\mc{Z}\in \mathbb{C}^{\textbf{N}(n)}+\nl(\mc{C}\otimes\mc{D}^*)$.
\end{corollary}

\begin{theorem}\label{thm5.1}
Let  $\mc{A} \in \mathbb{C}^{\textbf{N}(n) \times \textbf{N}(n) }$ be a core tensor and $\mc{B}\in\rg(\mc{A})$. Then the singular tensor equation
$\mc{A}\n\mc{X} = \mc{B}$, has unique solution in $\rg(\mc{A})$, and is given by 
$\mc{X} = \mc{A}^{\core}\n\mc{B}$.
\end{theorem}
\begin{proof}
Clearly $\mc{A}^{\core}\n\mc{B}\in\rg(\mc{A})$ since $\mc{A}^{\core}\n\mc{B}=\mc{A}\n(\mc{A}^{\core})^2\n\mc{B}$. By Lemma \ref{lm5.1}, $\mc{A}^{\core}\n\mc{B}$ is a solution in $\rg(\mc{A})$. Next we will show the uniqueness of the solution. Let $\mc{X} \in \rg(\mc{A})$ be any solution. This implies $\mc{X}-\mc{A}^{\core}\n\mc{B}\in\rg(\mc{A})$.  
Further, from Corollary \ref{lm5.2}, we obtain $\mc{X}-\mc{A}^{\core}\n\mc{B}=(\mc{I}-\mc{A}^{\core}\n\mc{A})\n\mc{Z}$ for some $\mc{Z}$. Now $\mc{A}\n(\mc{X}-\mc{A}^{\core}\n\mc{B})=(\mc{A}-\mc{A}\n\mc{A}^{\core}\n\mc{A})\n\mc{Z}=\mc{O}$. Thus $\mc{X}-\mc{A}^{\core}\n\mc{B}\in\nl(\mc{A})$. Hence $\mc{X}-\mc{A}^{\core}\n\mc{B}\in \rg(\mc{A})\cap \nl(\mc{A}) = \{\mc{O}\}$. Therefore, the solution  $\mc{A}^{\core}\n\mc{B}$ is unique.
\end{proof}

\begin{theorem} \label{sythm}
Let $\mc{A} \in \mathbb{C}^{\textbf{N}(n) \times \textbf{N}(n)},~
    \mc{X} \in \mathbb{C}^{\textbf{N}(n) \times \textbf{N}(n)},~
    \mc{B}  \in \mathbb{C}^{\textbf{N}(n) \times \textbf{N}(n)}$ and
   $\mc{D} \in \mathbb{C}^{\textbf{N}(n) \times \textbf{N}(n)}$. Then the tensor equation
$\mc{C}\n\mc{X}\n\mc{D} = \mc{B}$;
\begin{itemize}
\item[(a)] is solvable if and only if  there exist $\mc{C}^{\core}$ and $\mc{D}^{\core}$
such that $$\mc{C}\n\mc{C}^{\core}\n\mc{B}\n\mc{D}^{\core}\n\mc{D} =
\mc{B},$$
\item[(b)] in this case, the general solution is
\begin{equation}\label{eqth}
\mc{X} = \mc{C}^{\core}\n\mc{B}\n\mc{D}^{\core} + \mc{Z} -
\mc{C}^{\core}\n\mc{C}\n\mc{Z}\n\mc{D}\n\mc{D}^{\core},
\end{equation}
where $\mc{Z}\in \mathbb{C}^{\textbf{N}(n) \times \textbf{N}(n)}$ is an arbitrary tensor.
\end{itemize}
\end{theorem}

\begin{proof}
$(a)$ Let $\mc{C}\n\mc{C}^{\core}\n\mc{B}\n\mc{D}^{\core}\n\mc{D} =
\mc{B}$. and $\mc{Y}=\mc{C}^{\core}\n\mc{B}\n\mc{D}^{\core}$. Now 
$\mc{C}\n\mc{Y}\n\mc{D}=\mc{C}\n\mc{C}^{\core}\n\mc{B}\n\mc{D}^{\core}\n\mc{D}$. Thus $\mc{Y}$ is a solution of $\mc{C}\n\mc{X}\n\mc{D} = \mc{B}$. Conversely, if $\mc{U}$ is a solution of  $\mc{C}\n\mc{X}\n\mc{D} = \mc{B}$, then  
\begin{eqnarray*}
\mc{B}&=&\mc{C}\n\mc{U}\n\mc{D}=\mc{C}\n\mc{C}^{\core}\n\mc{C}\n\mc{U}\n\mc{D}\n\mc{D}^{\core}\n\mc{D}=\mc{C}\n\mc{C}^{\core}\n\mc{B}\n\mc{D}^{\core}\n\mc{D}.
\end{eqnarray*}
$(b)$ Suppose the multilinear system $\mc{C}\n\mc{X}\n\mc{D} = \mc{B}$ is solvable. Let $\mc{X} = \mc{C}^{\core}\n\mc{B}\n\mc{D}^{\core} + \mc{Z} -
\mc{C}^{\core}\n\mc{C}\n\mc{Z}\n\mc{D}\n\mc{D}^{\core}$.  Since $\mc{C}\n\mc{X}\n\mc{D}=\mc{C}\n( \mc{C}^{\core}\n\mc{B}\n\mc{D}^{\core} + \mc{Z} -
\mc{C}^{\core}\n\mc{C}\n\mc{Z}\n\mc{D}\n\mc{D}^{\core})\n\mc{D}\\
=\mc{B}+\mc{C}\n\mc{Z}\n\mc{D}- \mc{C}\n\mc{C}^{\core}\n\mc{C}\n\mc{Z}\n\mc{D}\n\mc{D}^{\core})\n\mc{D}=\mc{B}$. Thus $\mc{X}$ is a solution of \eqref{sy3}. Further consider $\mc{Y}$ be any arbitrary solution of solution of  \eqref{sy3}. Since
\begin{equation*}
  \mc{Y}=\mc{C}^{\core}\n\mc{B}\n\mc{D}^{\core}+\mc{Y}-\mc{C}^{\core}\n\mc{B}\n\mc{D}^{\core}\\
=\mc{C}^{\core}\n\mc{B}\n\mc{D}^{\core}+\mc{Y}-\mc{C}^{\core}\n\mc{C}\n\mc{Y}\n\mc{D}\n\mc{D}^{\core},  \end{equation*}
so the general solution is of the form \eqref{eqth}.
\end{proof}

In order to show how the core inverse of tensors are employed in the two dimensional Poisson problem in the multilinear system framework.  We present an example to illustrate our result.

\begin{example}\label{pde1}
Consider the following partial differential equation
\begin{equation}\label{PDE11}
    \frac{\partial^2 u}{\partial x^2}+\frac{\partial^2 u}{\partial y^2}=f(x,y),~(x,y)\in \Omega=[0,1]\times [0,1]
\end{equation}
 with Neumann boundary conditions. If we apply 5-point stencil central difference scheme on a uniform grid with $m^2$ nodes, we obtain the following tensor equation 
\begin{figure}[t!]\label{exa-1}
\begin{center}
\begin{tabular}{cc}
\subfigure[] { \includegraphics[width=0.45\textwidth]{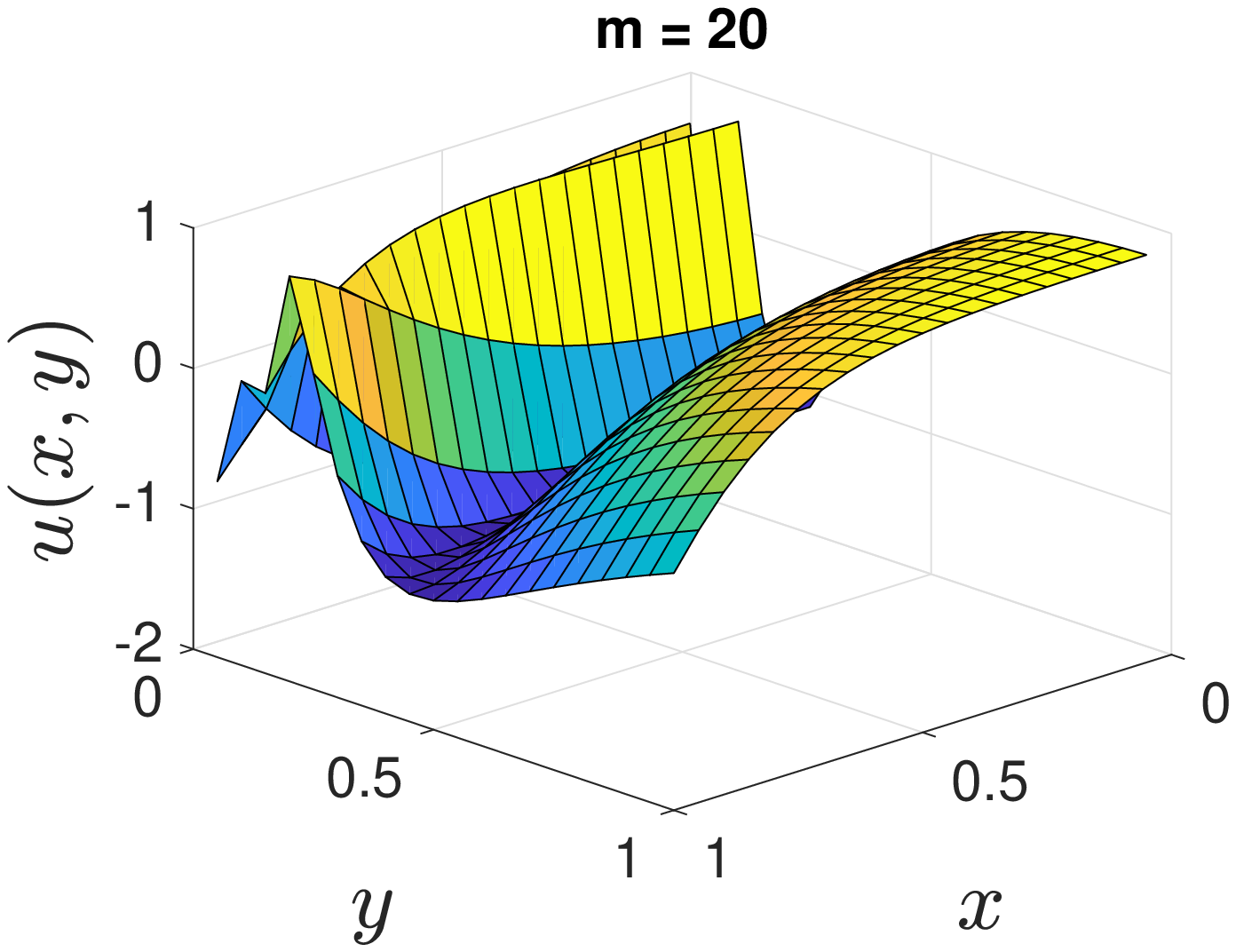}}&
\subfigure[] { \includegraphics[width=0.45\textwidth]{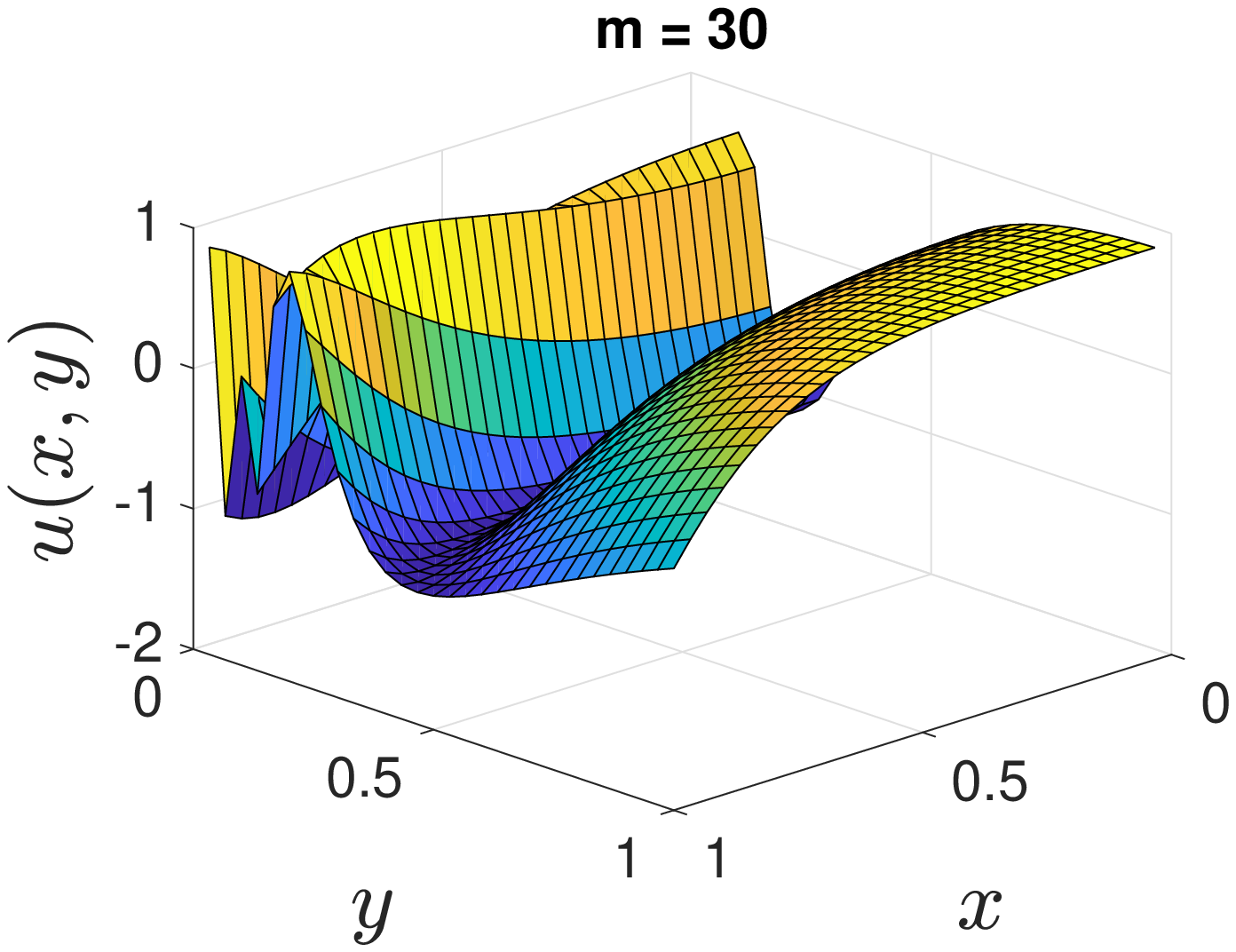}}\\
\subfigure[] { \includegraphics[width=0.45\textwidth]{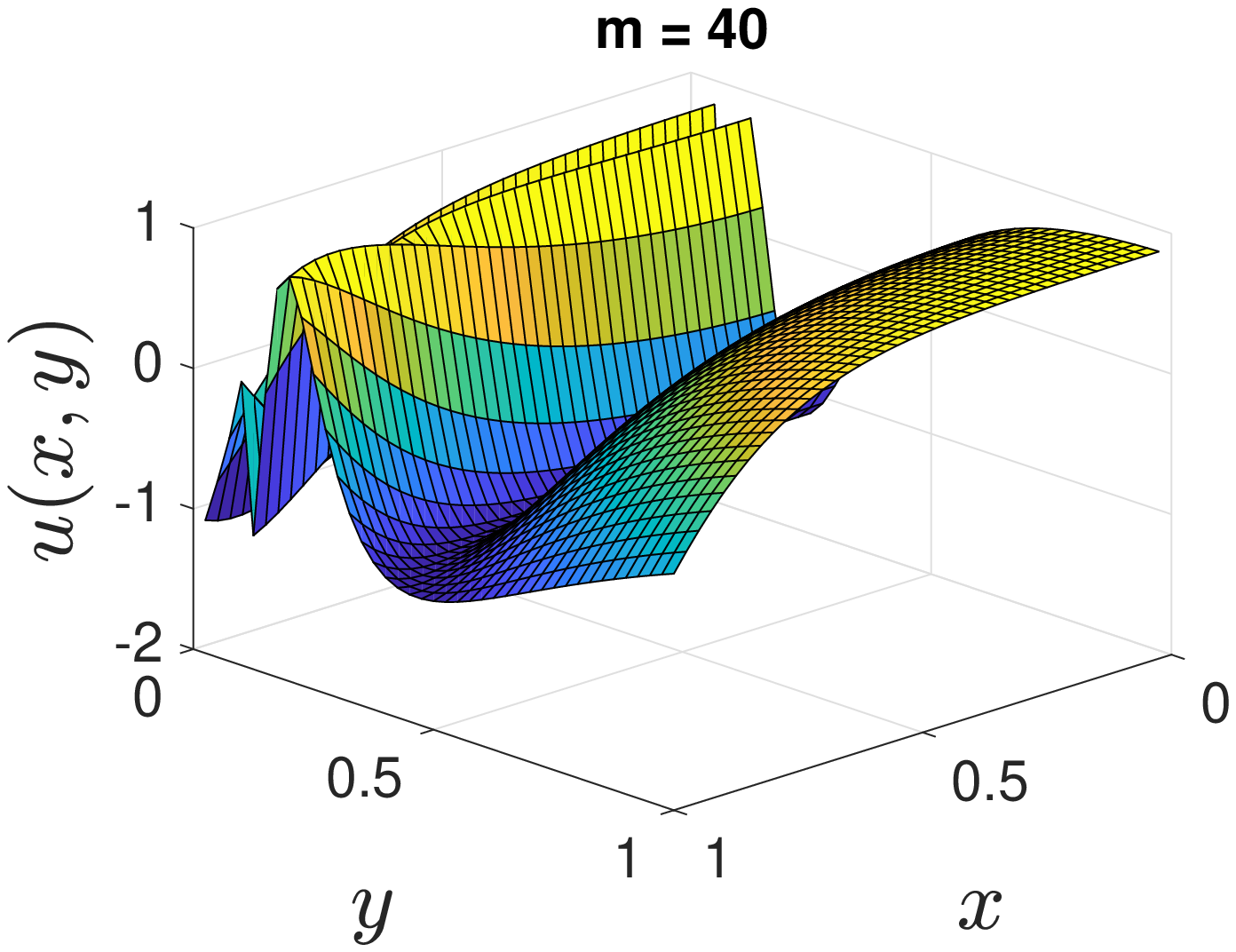}}&
\subfigure[] { \includegraphics[width=0.45\textwidth]{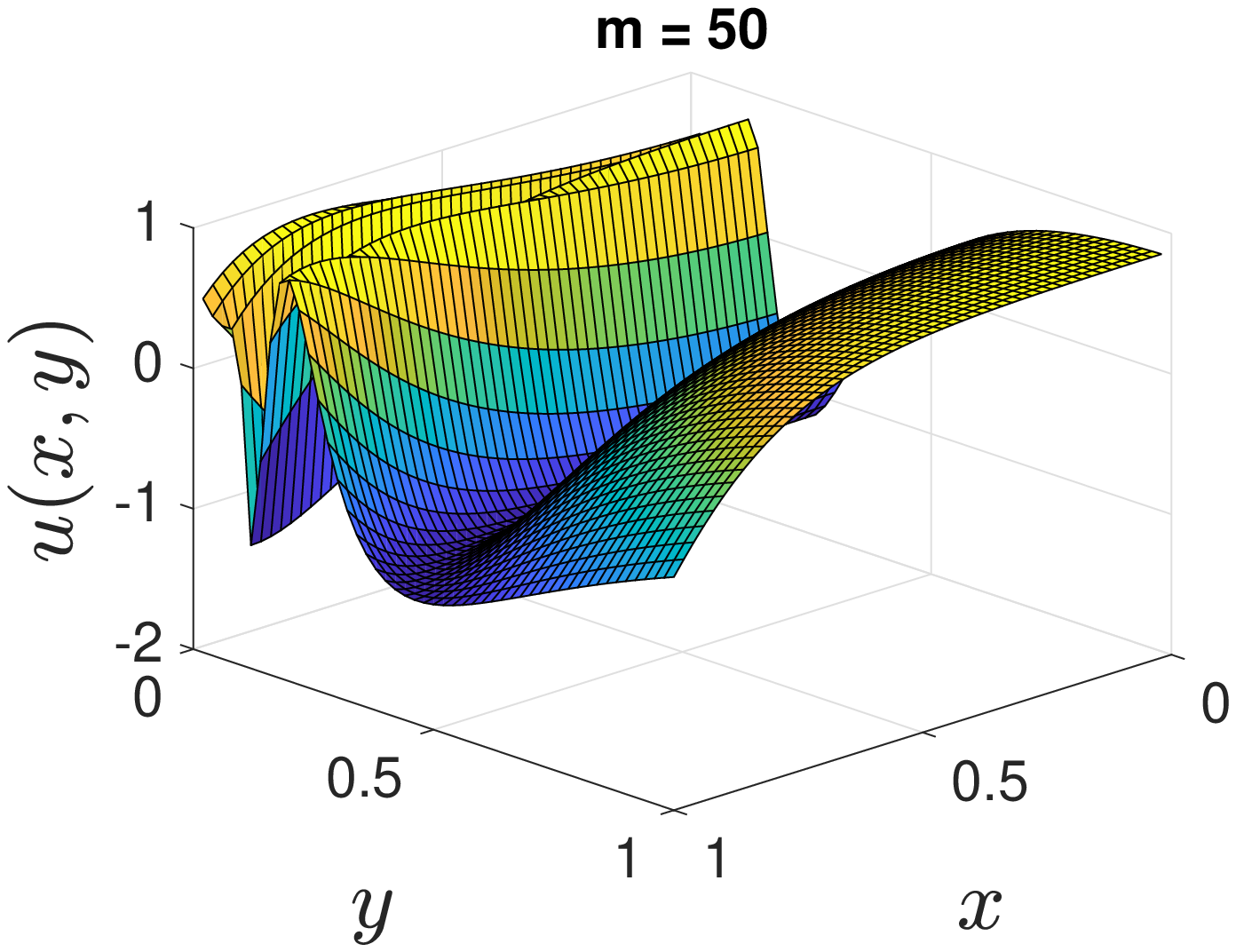}}
\end{tabular}
\caption{Solution of the multilinear system for different values of $m$.}
\label{finalbell}
\end{center}
\end{figure}
\begin{equation}\label{1stexam}
     \mc{A}*_2\mc{X}=\mc{B}, ~\mc{X}=(u_{kl})\in\mathbb{R}^{m\times m} \textnormal{~and~}  \mc{B}=(b_{ij})\in\mathbb{R}^{m\times m}, 
 \end{equation}
and the tensor $\mc{A}=(a_{ijkl})\in\mathbb{R}^{m\times m\times m\times m}$ is of the form 
\begin{equation}\label{F-ex-ten}
\mc{A}=\mc{I}_m\kronecker \mc{P} +\mc{Q}\kronecker \mc{I}_m+\mc{D},
\end{equation}
where $\mc{I}_m \in \mathbb{R}^{m\times m}$ is the second order identity tensor. The second order tensors $\mc{P}\in\mathbb{R}^{m\times m}$ and $\mc{Q}\in\mathbb{R}^{m\times m}$ are of the form
\begin{equation*}
    \mc{P}=tridiagonal\left(-1,0,-1\right) = \left( \begin{array}{cccc}
0 & -1 & & 0\\
-1 & \ddots & \ddots & \\
& \ddots & \ddots & -1 \\
0 & & -1 & 0 \end{array} \right)=
 \mc{Q}.
\end{equation*}
Further, the tensor $\mc{D}\in\mathbb{R}^{m\times m\times m\times m}$ is a diagonal tensor, where the diagonal elements will change with respect to number of grid points.   
From the representation \eqref{F-ex-ten} (the coefficient tensor $\mc{A}$), it is clear that $ind(\mc{A}) = 1$. Thus the solution of the multilinear system \eqref{1stexam} becomes $\mc{X}=\mc{A}^{\core}*_2\mc{B}.$  
We consider a tensor $\mc{B}$ from $\rg(A),$ and calculate the approximate solution of the partial differential equation \eqref{PDE11} with  different choices of $m$, which are presented in Figure 1.
\end{example}

\section{Conclusion}
In this paper, we have studied reverse-order law and mixed-type of reverse-order law for core inverse of tensors. We have also explored core inverse solutions of multilinear systems of tensors via the Einstein product and discussed two-dimension Poisson problems in the multilinear system framework.

\noindent {\bf{Acknowledgments}}\\
This research work was supported by Science and Engineering Research Board (SERB), Department of Science and Technology, India, under the Grant No. EEQ/2017/000747.

\bibliographystyle{abbrv}
\bibliographystyle{vancouver}
\bibliography{Reversecore}
\end{document}